\documentclass[11pt]{article}
\usepackage{etex}
\usepackage[a4paper]{geometry}

\usepackage{amsmath}
\usepackage{amssymb}
\usepackage{amsthm}
\usepackage{mathrsfs}
\usepackage{bbm}
\usepackage{empheq}
\usepackage[loose]{subfigure}
\usepackage{epsfig}
\usepackage{graphicx}
\usepackage{psfrag}
\usepackage[usenames,dvipsnames]{pstricks}
\usepackage{pst-plot}
\usepackage[colorlinks]{hyperref}
\usepackage{tabls}
\usepackage{paralist}
\usepackage{hyperref}
\usepackage[all]{xy}
\DeclareSymbolFontAlphabet{\Bbb}{AMSb}

\newlength{\fixboxwidth}
\newcommand{\COMMENT}[1]{}

\newcommand{\R}{\mathbb{R}}

\newcommand{\quark}{\setbox0\hbox{$x$}\hbox to\wd0{\hss$\cdot$\hss}}

\newcommand{\s}{\sigma}

\newtheorem{thm}{Theorem}[section]
\newtheorem{prop}[thm]{Proposition}
\newtheorem{lem}[thm]{Lemma}

\newtheorem{cor}[thm]{Corollary}
\theoremstyle{definition}
\newtheorem{defn}[thm]{Definition}
\newtheorem{rmk}[thm]{Remark}

\hypersetup{
    linkcolor=blue,
}

\title{Qualitative Robustness in Bayesian Inference
}
\author{Houman Owhadi and Clint Scovel
\\
California Institute of Technology
}
\date{\today}

\makeatletter
\@addtoreset{equation}{section}
\@addtoreset{figure}{section}
\@addtoreset{table}{section}
\makeatother

\renewcommand{\thefigure}{\arabic{section}.\arabic{figure}}

\makeatletter
\renewcommand{\p@subfigure}{\thefigure}
\makeatother

\newcounter{mycount}

\begin{document}
\maketitle
\begin{abstract}
The practical implementation of Bayesian inference requires numerical approximation when closed-form expressions are not available. What types of accuracy (convergence) of the numerical approximations
 guarantee robustness and what types do not? In particular, is the recursive application of Bayes'
 rule robust when subsequent data or posteriors are approximated? When the prior is the push forward of a distribution by the map induced by the solution of a PDE, in which norm should that solution be approximated?
Motivated by such questions, we investigate the sensitivity of  the  {\em distribution} of posterior distributions
 (i.e.~of {\em posterior distribution}-valued random variables, randomized through the data) with respect to perturbations of the prior and data generating distributions in the limit when the number of data points grows towards infinity.
\end{abstract}

	

\section{Introduction and motivations}
When we apply Bayesian inference with Gaussian priors and linear observations, we do not actually compute  Bayes' rule but solve the linear system whose solution is the conditional expectation, and robustness is guaranteed by that of our linear solver.
This paper is motivated by robustness questions that  arise in the numerical implementation of Bayes' rule for continuous systems when closed-form expressions are not available for the computation of posterior values/distributions.
 For example,  what is the sensitivity of posterior values or the sensitivity of the distribution of posterior values to perturbations introduced by the numerical approximation of the prior?
 When  Bayes' rule is applied in a recursive manner and approximated posterior distributions are used as prior distributions or approximated data is used in the conditioning process, do we have robustness guarantees on subsequent posterior distributions/values? Is it possible to numerically approximate the optimal prior/mixed strategy of a decision theory problem (arising in the continuous setting under the complete class theorem \cite{Wald:1950}) when closed form expressions are not available, and if it is, in which metric should we do so?

	\begin{figure}[tp]
\begin{center}
			\includegraphics[width=\textwidth]{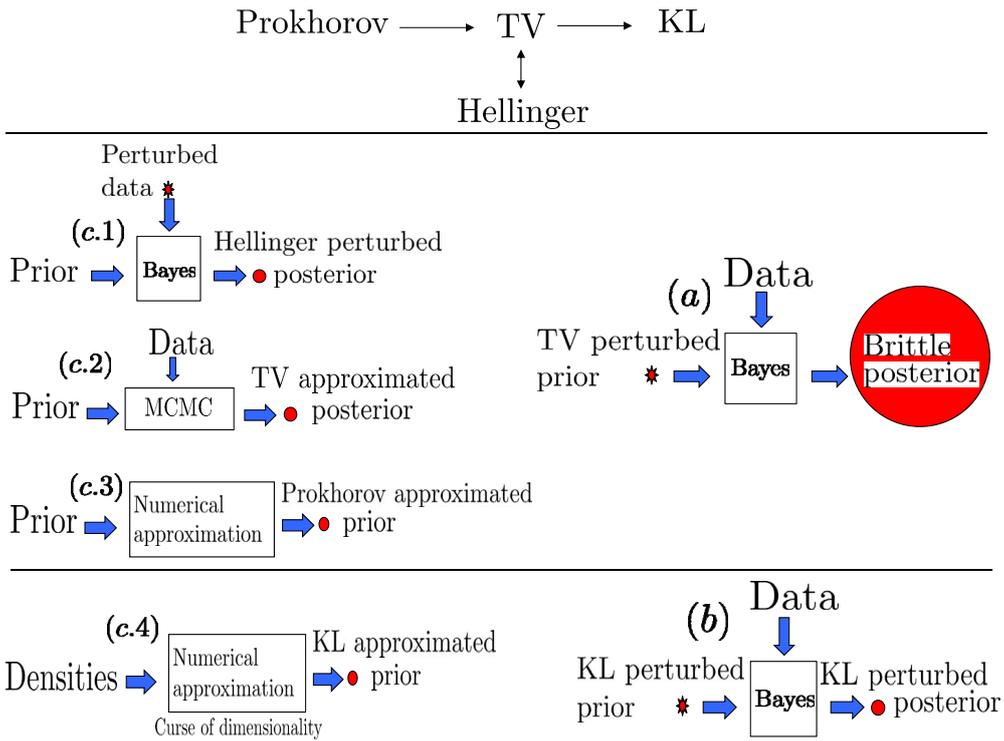}
		\caption{On the robustness of the Bayes' rule.
 The top panel illustrates relationships among probability metrics and,  as in \cite{Gibbs_Su},
  a directed arrow from A to B means that (for some function $h$) $d_A\leq h(d_B)$. The central panel illustrates
the generation of
non-robustness  and the bottom panel illustrates the generation of robustness.  }
		\label{figdragons}
\end{center}
	\end{figure}
 Figure  \ref{figdragons} provides an illustration of partial answers to such questions, which, when combined, can be used as a map to navigate robustness/non-robustness questions/issues arising from numerical approximations.
To begin, (a)
 \cite{OSS:2013,OwhadiScovel:2013,owhadi_shatters} shows that the range of posterior values of the quantity of interest under perturbations of the prior in Prokhorov or TV metrics is the deterministic range of the quantity of interest (we will refer to this maximal sensitivity property as brittleness).
Moreover, (c.1)  \cite{Stuart:2010} shows that the posterior distributions are controlled by the Hellinger metric under the application of the Bayes rule with an exact prior under the perturbation of a finite number of
 finite dimensional samples (see also \cite{ThanhGhattas14}).
  Since the Hellinger and TV metrics are equivalent \cite{Gibbs_Su},  (c.1) and (a) suggest that, without specific restrictions, it is, in general, not possible to offer guarantees on the robustness of recursive Bayes under perturbations of the data.
 Similarly (c.2) the convergence of Markov chains (on continuous state spaces) used in MCMC algorithms (such as the Metropolis algorithm) is generallly analyzed with respect to the TV topology \cite{Roberts2004, Gibbs04}. However
 it should be noted that, according to
Roberts and Rosenthal \cite[Pg.~10]{Roberts_mcmc}, Gibbs \cite{Gibbs04},  Gelman \cite[Intro]{Gelman_inference} and Madras and Sezer
 \cite[Intro]{Madras_quantitative}, convergence in TV is in general not guaranteed, so that one might say in general
that convergence of MCMC is  {\em at best} in TV.
 Therefore combining (c.2) and (a) suggests that, without further restrictions, it is, in general, not possible to offer guarantees on the robustness of recursive Bayes if posteriors are approximated using MCMC.
Moreover, in many cases,
 the prior may need to undergo a numerical approximation/discretization step prior to conditioning.
For example, in popular applications of Bayes' rule to stochastic PDEs \cite{Woodbury00, DashtiStuart2011pde} one pushes forward the prior from the space of coefficients of the PDE to the solution space where it is conditioned.
 Consequently, if
 the PDE is numerically approximated this implies an approximation of the pushed-forward prior. Therefore
 (c.3) representing  numerical discretization/approximation as a continuous map between two Polish spaces, it is known that the push forward of a measure under such maps is continuous in the weak topology, which is metrized by the Prokhorov distance. Therefore, combining (c.3) and (a) suggests that, without further restrictions, it is, in general, not possible to offer robustness guarantees to perturbations caused by the numerical discretization of the prior.

For positive results on the other hand, observe that
(b)  \cite{Gustafson_local} shows that posterior values (given a finite number of data points) are robust to approximation errors of the prior measured in Kullback-Leibler divergence. Therefore if (c.4) the numerical approximation map is continuous in  Kullback-Leibler divergence then posterior values (given a finite number of data points) are robust to that numerical discretization step. However observe that unless closed form expressions are available, one must keep track of densities to achieve the continuity of the numerical approximation map in Kullback-Leibler divergence, which is a task plagued by the curse of dimensionality.

Since the brittleness of posterior distributions and values with respect perturbations of the prior defined in TV or Prokhorov metrics
is an obstacle to obtaining robustness guarantees to numerical approximation errors when closed form expressions are unavailable, it is natural to ask whether robustness could be guaranteed by considering the {\em distribution} of posterior distributions or posterior values generated by the random generation of the data. To answer this question we develop a framework for quantifying the sensitivity of  the distribution of posterior distributions with respect to perturbations of the prior and data generating distributions in the limit when the number of data points grows towards infinity. In this generalization of  Cuevas' \cite{Cuevas}
extension  of Hampel's \cite{Hampel} notion of {\em qualitative robustness} to Bayesian inference
to include   both perturbations of the prior and the data generating distribution, posterior distributions are analyzed as measure-valued random variables (measures randomized through the data) and their robustness is quantified using the total variation, Prokhorov, and Ky Fan metrics.
Our results  show  that (1) the assumption that the prior has Kullback-Leibler support at the parameter value generating the data, classically used to prove consistency, can also be used to prove the non-robustness of posterior distributions with respect to infinitesimal perturbations (in TV) of the class of priors satisfying that assumption, (2) for a prior which has global Kullback-Leibler support on a  space which is not totally bounded, we can establish non-robustness   and (3)  consistency, and the unstable
nature of the conditions which generate it, produces non-robustness, and a careful selection of the prior is important if both properties (or their approximations) are to be achieved.
 The mechanisms supporting our results  are different and complementary to those discovered by Hampel   and developed by Cuevas.
To obtain them, we derive in Section \ref{sec_schwartz} a corollary to
 Schwartz' consistency Theorem that leads to robustness or non-robustness results depending on the topology defining the
continuity of the numerical approximation map (Kullback-Leibler, TV or Prokhorov). Moreover, this corrollary
is further developed in Proposition \ref{prop_schwartz} to analyze the convergence of random measures in
the qualitative robustness framework.
 More precisely,
although in \cite{Gustafson_local} it is shown that the Fr\'{e}chet derivatives of posterior values to Kullback-Leibler perturbations of the prior  may diverge to infinity with the number of data points, a simple application of Proposition \ref{prop_schwartz} implies the robustness of the distribution of posterior distributions/values under Kullback-Leibler perturbations of the prior in the limit where the number of data points goes to infinity. On the other hand,
  application of Proposition \ref{prop_schwartz} also suggest the lack of robustness of the distribution of posterior distributions/values to  TV or Prokhorov perturbations of the prior, where we note that Prokhorov perturbations include classes of perturbations defined by generalized moment constraints, as in \cite{ Betro94,Betro96,Betro00,Betro09, OSS:2013,OwhadiScovel:2013,owhadi_shatters}.

\section{Qualitative Robustness for Bayesian Inference}
Hable and Christmann \cite{Hable} have recently established
qualitative robustness for support vector machines. Consequently, it appears natural to inquire
into the qualitative robustness of Bayesian inference.
Hampel \cite{Hampel}  introduced the notion of the qualitative robustness of a sequence of estimators
and Cuevas \cite{Cuevas} has extended Hampel's definition and his basic structural results to Polish  spaces. Since the space $\mathcal{M}(\Theta)$ of priors and posteriors equipped with the weak topology
 is Polish whenever $\Theta$ is, Cuevas' extension has direct applications to Bayesian inference.
Boente et al.~\cite{Boente} have developed  qualitative robustness for stochastic processes,
 Nasser et al.~\cite{Nasser}
   for  estimation, and
 Basu et al.\cite{Basu_stability} for Bayesian inference with a single sample.
   The notion of qualitative robustness  introduce
 in this paper is
a straightforward
 generalization of that introduced by Hampel \cite{Hampel} and developed by
 Cuevas \cite{Cuevas}, see also  Cuevas \cite{gonzalez1984definicion}. Indeed,
 this version requires no introduction of loss and risk functions and concerns itself with not just  expected values but with the full distribution of the effects
of the randomness of the observations.  It considers a fixed model so is not concerned with
robustness with respect to model specification, although such considerations can easily be included. Moreover,
since it is formulated with respect to classic notions in probability,
 it appears to us as simple, natural, flexible, without calibration issues, and easy to interpret
  statistically.

Metrics on spaces of measures  and random variables will be important in its formulation. Fortunately,
this
 is a well studied field, see e.g.~Rachev et al.~\cite{Rachev:etal} and Gibbs and Su \cite{Gibbs_Su}, but to keep the presentation simple, here we will restrict our attention to the total variation,  Prokhorov
and Ky Fan metrics (we refer to  \cite{ Betro94,Betro96,Betro00,Betro09} for motivations for considering classes of priors defined in the Prokhorov metric).
For measurable spaces $\Theta$ and $X$, we write $\mathcal{M}(\Theta)$ and  $\mathcal{M}(X)$
for the set of probability distributions on $\Theta$ and  $X$ respectively.
In this  general setting,
we can metrize the spaces of measures $\mathcal{M}(X)$ and $\mathcal{M}(\Theta)$
using total variation. This latter metrization makes $\mathcal{M}(\Theta)$ into a topological space
whose Borel structure can be used to define
 the space
$\mathcal{M}^{2}(\Theta)$ of probability measures on $\mathcal{M}(\Theta)$, which can also be metrized
using the total variation. However, the separability of these spaces will be extremely useful for us and,
 in general, these spaces will not be separable under the total variation metric.
Recall that
for a metric space $(S,d)$, the Prokhorov metric $d_{Pr}$ on the space
$\mathcal{M}(S)$ of Borel probability measures is defined
by
\begin{equation}
\label{def_Pr}
  d_{Pr}(\mu_{1},\mu_{2}):=\inf{\bigl\{\epsilon: \mu_{1}(A) \leq \mu_{2}(A^{\epsilon})+\epsilon,\, A \in \mathcal{B}(S)\bigr\} },
\quad \mu_{1},\mu_{2}\in \mathcal{M}(S)\, ,
\end{equation}
where
\[ A^{\epsilon}:=\{x' \in S: d(x,x') < \epsilon \, \, \, \text{for some} \, \, \,  x \in A\} \,  .\]
According to Dudley \cite[Thm.~11.3.1]{Dudley:2002}, the Prokhorov metric is a metric
on $\mathcal{M}(S)$.
 Consequently, when $\Theta$ and $X$  are metric, ,
we can also metrize the spaces of measures $\mathcal{M}(\Theta)$
 and  $\mathcal{M}(X)$  with the Prokhorov metrics,
and  having done so we can
define the space
 $\mathcal{M}^{2}(\Theta):=\mathcal{M}(\mathcal{M}(\Theta))$
of Borel probability measures on the metric space
  $(\mathcal{M}(\Theta),d_{Pr\Theta})$ of Borel probability measures
on $\Theta$ and metrize it with the Prokhorov metric $d_{Pr^{2}\Theta}$.
 Furthermore, when
$(S,d)$ is a separable metric space,
Dudley \cite[Thm.~11.3.3]{Dudley:2002} asserts that
the Prokhorov metric metrizes weak convergence and
 Aliprantis and Border
  \cite[Thm.~15.12]{AliprantisBorder:2006}  asserts that the metric space
$(\mathcal{M}(S),d_{Pr})$ is separable.

 Therefore,
 when $X$ and $\Theta$ are
separable metric spaces,
the Prokhorov metrics $d_{PrX}$ and $d_{Pr\Theta}$ metrize weak convergence
 in $\mathcal{M}(X)$ and $\mathcal{M}(\Theta)$ respectively and both
metric spaces $(\mathcal{M}(X),d_{PrX})$  and $(\mathcal{M}(\Theta),d_{Pr\Theta})$ are separable.
Consequently, when $\Theta$ is a separable metric space,  $(\mathcal{M}(\Theta),d_{Pr\Theta})$ is a separable metric space and therefore
 $(\mathcal{M}^{2}(\Theta),d_{Pr^{2}\Theta})$  is a separable metric space.

The separability of $(\mathcal{M}(\Theta),d_{Pr\Theta})$ is sufficient to define
the Ky Fan metric on a space of $(\mathcal{M}(\Theta)$-valued random variables.
Indeed, for a separable metric space $S$, probability space $(\Omega, \Sigma,P)$,
and two  $S$-valued random variables $Z:\Omega \rightarrow S$ and $W:\Omega \rightarrow S$,    the Ky Fan distance between $Z$ and $W$,
 see e.g.~Dudley \cite[Pg.~289]{Dudley:2002}, is defined as
\begin{equation}
\label{def_kyfan}
\alpha (Z,W):=\inf\bigl\{\epsilon \geq 0: P(d(Z,W) >\epsilon) \leq \epsilon\bigr\}\, .
\end{equation}
By Dudley \cite[Thm.~9.2.2]{Dudley:2002},  the Ky Fan metric is a metric on
the space of  $S$-valued  random variables from $(\Omega,\Sigma,P)$ and
 metrizes convergence in probability for them.
  Consequently, when $\Theta$ is a separable metric space and
 $\mathcal{M}(\Theta)$ is metrized with the Prokhorov metric $d_{Pr\Theta}$,
 the Ky Fan  metric $\alpha$ of \eqref{def_kyfan} metrizes the space of $\mathcal{M}(\Theta)$-valued random variables
$Z:(X^{\infty},\mu^{\infty}) \rightarrow (\mathcal{M}(\Theta),d_{Pr\Theta})$ for each
$\mu\in \mathcal{M}(X)$. Since this family of metrics depends on the measure $\mu$ we indicate
this dependence by writing $\alpha_{\mu}$.
Moreover, when  $\Theta$ and $X$ are separable metric spaces,
their Borel $\s$-algebras are countably generated,
which is required  to apply Doob's Theorem to assert that a domninated measurable
model has a jointly measurable family of densities, which is required in the consistency theorem
of Schwartz which we will need.

When $\Theta$ and $X$ are Borel subsets of Polish metric spaces,
  they are separable metric spaces so the above applies. Let us now show that
  the assumption also facilitates the {\em  measurability} of Bayesian conditioning  that will be needed to
 define its qualitative robustness.   To that end,
from now on let us  place as default the weak topologies on  $\mathcal{M}(X)$, $\mathcal{M}(\Theta)$
and $\mathcal{M}^{2}(\Theta)$
and metrize them with the Prokhorov metrics $d_{PrX}$, $d_{Pr\Theta}$ and $d_{Pr^{2}\Theta}$. This is primarly to obtain
well-defined Bayesesian conditioning while at the same time applicability of Schwartz's consistency theorem.
We will also place other metric structures on $\mathcal{M}(X)$,  $\mathcal{M}(\Theta)$  and  $\mathcal{M}^{2}(\Theta)$ to quantify
the size of perturbations
and indicate them with the notation $d_{\mathcal{M}(X)}$, $d_{\mathcal{M}(\Theta)}$ and
 $d_{\mathcal{M}^{2}(\Theta)}$.
  Consider  a
 measurable model $P:\Theta \rightarrow \mathcal{M}(X)$. Since  Aliprantis and Border
 \cite[Thm.~15.13]{AliprantisBorder:2006} implies that  the map $\mathcal{M}(X)\rightarrow \R$ defined by
$\mu \mapsto \mu(A)$ is Borel measurable for all $A \in \mathcal{B}(X)$,   it follows that
$P$ corresponds to a Markov kernel.  Consider a prior $\pi \in \mathcal{M}(\Theta)$.  Then since
$\Theta$ is assumed to be a Borel subset of a Polish space, it follows from
 Schervish \cite[Thm.~B.46]{Schervish_theory} that there exists
a family $\pi_{x}, x\in X$  of conditional probability measures generated by the model $P$  such that
the map $x \mapsto  \int_{\Theta}{fd\pi_{x}}, x \in X$ is $\mathcal{B}(X)$-measurable for all
bounded and measurable functions $f:\Theta \rightarrow \R$.
Note that after the proof Schervish mentions that such a family of conditional measures is not unique.
 Since both $\Theta$ and $X$ are separable and metrizable,
it then follows from  Aliprantis and Border
 \cite[Thm.~19.7]{AliprantisBorder:2006} that the resulting map
$x \mapsto  \pi_{x}$ from $X$ to $\mathcal{M}(\Theta)$ is measurable.
For multiple samples,   it is clear that $X^{n}$
 is a Borel subset of the $n$-th power  of the ambient Polish space of $X$.
By Billingsly's \cite[Thm.~2.8]{Billingsley1} characterization of weak convergence on product spaces
it follows that
    the injection $\mathcal{M}(X) \rightarrow  \mathcal{M}(X^{n})$ defined by
$\mu \mapsto  \mu^{n}$ is continuous,
so that it follows that $P^{n}:\Theta \rightarrow  \mathcal{M}(X^{n})$,
defined by $P^{n}(\theta)=P(\theta)^{n}, \theta \in \Theta$,
 is measurable and therefore, by the same arguments as above, we obtain  a family of  multisample
conditional measures  $\pi_{x^{n}}, x^{n}\in X^{n}$  such that the resulting map
\[\bar{\pi}:X^{n}\rightarrow  \mathcal{M}(\Theta)\] defined by the determination of the posteriors
 \begin{equation}
\label{def_barpi}
\bar{\pi}(x^{n}):=\pi_{x^{n}},\quad x^{n} \in X^{n}
\end{equation}
is measurable.
Therefore,
its corresponding pushforward operator
\[\pi_{*}:\mathcal{M}(X^{n}) \rightarrow \mathcal{M}^{2}(\Theta)\, \]
is well-defined,
 where we have removed the bar over $\pi$ in the notation to emphasize that this pushforward
operator $\pi_{*}$ corresponds to the prior $\pi$.
Then to consider how the posteriors $\pi_{x^{n}}$  vary as a function of the sample data $x^{n}$
when it is generated by i.i.d.~sampling from $\mu$,
 since  $\mathcal{M}^{n}(X) \subset \mathcal{M}(X^{n}) $  it follows that
$\mu^{n} \in \mathcal{M}(X^{n}) $ so
we can utilize the pushforward operator $\pi_{*}$ to define
\begin{equation*}
\label{def_push}
  \pi_{*}\mu^{n} \in \mathcal{M}^{2}(\Theta)
\end{equation*}
the sampling distribution of the posterior distribution $\pi_{x^{n}}$ when
$x^{n}\sim \mu^{n}$.

For a fixed
 prior $\pi \in \mathcal{M}(\Theta)$, we say that the Bayesian inference
 is qualitatively robust at a data generating distribution $\mu\in \mathcal{M}(X)$  with respect to an admissible set $\mathcal{P}$ containing
 $\mu$,
and metrics $d_{\mathcal{M}(X)}$
and $d_{\mathcal{M}^{2}(\Theta)}$ on  $\mathcal{M}(X)$ and $\mathcal{M}^{2}(\Theta)$,   if for any $\epsilon > 0$, there exists a $\delta >0$ such that
\[  \acute{\mu} \in \mathcal{P},\,\, d_{\mathcal{M}(X)}(\mu,\acute{\mu}) < \delta \implies d_{\mathcal{M}^{2}(\Theta)}
(\pi_{*}\mu^{n},\pi_{*}\acute{\mu}^{n}) <\epsilon\]
for  large enough $n$.
On the other hand,  when the data generating distribution $\mu$ is fixed and we vary the prior $\pi$,
we consider the sequence
 of maps
\[\pi_{n}:X^{\infty} \rightarrow \mathcal{M}(\Theta)\]
defined by
\begin{equation}
\label{eq_dekjjibib}
\pi_{n}(x^{\infty}):=\pi_{x^{n}},\quad x^{\infty} \in X^{\infty}\,.
\end{equation}
Since the projection $P_{n}: X^{\infty}\rightarrow X^{n} $ is continuous and
$\pi_{n}=\bar{\pi}\circ P_{n}$, it follows from the measurability
of $\bar{\pi}:X^{n} \rightarrow \mathcal{M}(\Theta)$, that
$\pi_{n}$ is measurable, and therefore
 the resulting sequence
\[\pi_{n}:(X^{\infty},\mu^{\infty}) \rightarrow \mathcal{M}(\Theta)\, \]
is a sequence of  $\mathcal{M}(\Theta)$-valued random variables.
For each $\mu \in \mathcal{M}(X)$, let  $\alpha_{\mu}$  be a metric on the space  of
 $\mathcal{M}(\Theta)$-valued random variables
whose domain is the  probability space $(X^{\infty},\mu^{\infty})$.
Then,  for a prior $\pi \in \mathcal{M}(\Theta)$,
 we say that the Bayesian inference is qualitatively robust at $\pi$ with  respect to
an admissible set $\Pi\subset \mathcal{M}(\Theta)$ containing $\pi$  and metrics $\alpha_{\mu}$ and
$d_{\mathcal{M}(\Theta)}$ on $\mathcal{M}(\Theta)$
if
 given   $\epsilon>0$, there exists a $\delta >0$ such that
\[  \acute{\pi} \in \Pi,\,\, d_{\mathcal{M}(\Theta)}(\pi,\acute{\pi}) < \delta  \implies
\alpha_{\mu}(\pi_{n},\acute{\pi}_{n}) < \epsilon \,
\]
for large enough $n$.

These two definitions can be combined in a straightforward manner to define robustness corresponding
to a single prior/data generating pair. However,
to consider  a larger class of distributions than a single pair, we let
 $\mathcal{Z} \subset \bigl(\mathcal{M}(\Theta) \times \mathcal{M}(X)\bigr)^{2} $
denote the admissible set of prior-data generating distribution pairs  $\bigl((\pi,\mu),(\acute{\pi},\acute{\mu})\bigr)
\in  \bigl(\mathcal{M}(\Theta) \times \mathcal{M}(X)\bigr)^{2} $   such that
$(\pi,\mu) \in \mathcal{M}(\Theta) \times \mathcal{M}(X)$
is an admissible  candidate for robustness and
$(\acute{\pi},\acute{\mu}) \in \mathcal{M}(\Theta) \times \mathcal{M}(X)$ is an admissible candidate for its perturbation. In particular,
the projection $\mathcal{Z}_{1}\subset\mathcal{M}(\Theta) \times \mathcal{M}(X)$ denotes the set of admissible prior-data generating pairs.
Now combining  in a straightforward manner we obtain:
\begin{defn}
\label{def_qr_kyfan}
Let $X$ and $\Theta$ be Borel subsets of Polish metric spaces and let
$\mathcal{M}(X)$ and $\mathcal{M}(\Theta)$ be equipped with the weak topology metrized
by the Prokhorov metrics $d_{PrX}$ and $d_{Pr\Theta}$. Let
   $\mathcal{M}^{2}(\Theta):=\mathcal{M}(\mathcal{M}(\Theta))$ be the space
of Borel probability measures on the metric space
  $(\mathcal{M}(\Theta),d_{Pr\Theta})$ of Borel probability measures
on $\Theta$ equipped with its weak topology metrized by its Prokhorov metric
$d_{Pr^{2}\Theta}$. Consider perturbation pseudometrics $d_{\mathcal{M}(X)}$, $d_{\mathcal{M}(\Theta)}$ and
 $d_{\mathcal{M}^{2}(\Theta)}$  on $\mathcal{M}(X)$, $\mathcal{M}(\Theta)$ and
 $\mathcal{M}^{2}(\Theta)$ respectively and for each $\mu \in \mathcal{M}(X)$, let
   $\alpha_{\mu}$ be a pseudometric on  the space of $\mathcal{M}(\Theta)$-valued random variables
on the  probability space $(X^{\infty},\mu^{\infty})$.
Let
 $\mathcal{Z} \subset  \bigl(\mathcal{M}(\Theta) \times \mathcal{M}(X)\bigr)^{2} $
denote the admissible set of prior-data generating distribution pairs and
suppose that
$P:\Theta \rightarrow \mathcal{M}(X)$ is  measurable.
 Then the Bayesian inference is qualitatively robust with respect to $\mathcal{Z}$,
if
 given   $\epsilon_{1}, \epsilon_{2} >0$, there exists  $\delta_{1},\delta_{2} >0$ such that
\[ \bigl( (\pi,\mu), (\acute{\pi},\acute{\mu})\bigr) \in \mathcal{Z},\quad
  \,  d_{\mathcal{M}(\Theta)}(\pi,\acute{\pi}) < \delta_{1},
\,\, d_{\mathcal{M}(X)}(\mu,\acute{\mu}) < \delta_{2}
\]
\[   \implies
d_{\mathcal{M}^{2}(\Theta)}(\acute{\pi}_{*}\mu^{n},\acute{\pi}_{*}\acute{\mu}^{n}) < \epsilon_{1} \, , \quad
\alpha_{\mu}(\pi_{n},\acute{\pi}_{n}) < \epsilon_{2} \,
\]
for large enough $n$.
\end{defn}
Finite sample versions, as introduced
in Hable and Christmann \cite[Def.~2]{Hable_robustness}, are also available.
Note that unlike Hampel and Cuevas who require ``for all $n$" in their definitions, we  follow
 Huber \cite{HuberRonchetti:2009} and Mizera \cite{Mizera} in only requiring closeness ``for large enough $n$". The results
of this paper are applicable to both versions.
Of course the  relevance of the specific notion of qualitative robustness  used
 depends on the  perturbation metrics used.
 The results of this paper
apply to the case when
  $\alpha_{\mu}$  is the Ky Fan metric, metrizing convergence in probability
on  the space of $\mathcal{M}(\Theta)$-valued random variables with domain
the  probability space $(X^{\infty},\mu^{\infty})$ and
the $d_{\mathcal{M}(\Theta)}$ is any metric weaker than
 the total variation.

\section{Lorraine Schwartz' Theorem}
\label{sec_schwartz}
The fundamental mechanism generating non robustness for Bayesian inference will be its consistency.
The breakthrough in consistency for Bayesian inference is considered to be  Schwartz' theorem \cite[Thm.~6.1]{Schwartz:1965},
so we use it  as a model
for consistency and the conditions sufficient to generate it.
Stated
in Barron, Schervish and Wasserman \cite[Intro]{BarronEtAl:1999}, Wasserman \cite[Pg.~3]{Wasserman_asymptotic} and
 Ghosal, Ghosh and Ramamoorthi \cite[Cor.~1]{Ghosal_issues} for the nonparametric case,
for the parametric case
we will operate in
 Borel subsets of  Polish metric spaces.
 By Schervish  \cite[Thm.~B.32]{Schervish_theory},  regular conditional probabilities exist
for conditioning random variables with values in such a space. Moreover, when the parametric model
is a Markov kernel and is dominated by a $\s$-finite measure, then by the Bayes' Theorem for densities
Schervish \cite[Thm.~1.31]{Schervish_theory}  we have, in addition, that the Bayes' rule for densities determines
a  valid  family of densities for the
 regular conditional distributions.
We say that a model $P:\Theta \rightarrow \mathcal{M}(X)$  is dominated if there exists a $\s$-finite
Borel measure $\nu$ on $X$ such that $P_{\theta}  \ll \nu,\, \theta \in \Theta$.

Recall the Kullback-Leibler divergence $K$ between two measures $\mu_{1}$ and $\mu_{2}$ defined by
\[  K(\mu_{1},\mu_{2}):=\int{\log{\Bigl(\frac{d\mu_{1}}{d\nu}\Big/\frac{d\mu_{2}}{d\nu}\Bigr)}d\mu_{1}}\, ,
\]
where $\nu$ is any measure such that both $\mu_{1}$ and $\mu_{2}$ are absolutely continuous with respect to $\nu$.
It is well known that $K$ is nonnegative, and that  it is finite  only if $\mu_{1} \ll \mu_{2}$, and in that case
  $K(\mu_{1},\mu_{2})=\int{\log{\frac{d\mu_{1}}{d\mu_{2}}}d\mu_{1}}\, . $
From this we can define the Kullback-Leibler ball $K_{\epsilon}(\mu)$ of radius $\epsilon$ about $\mu \in \mathcal{M}(X)$ by
$K_{\epsilon}(\mu)=\{\mu'\in \mathcal{M}(X): K(\mu,\mu') \leq \epsilon\}$.
For a model $P:\Theta \rightarrow \mathcal{M}(X)$, there is the pullback to a function
$K$ on $\Theta$ defined by $K(\theta_{1},\theta_{2}):=K(P_{\theta_{1}},P_{\theta_{2}})$ and when the model is dominated
 by a $\s$-finite measure $\nu$, if we let
$p(x|\theta):=\frac{dP_{\theta}}{d\nu}(x), x \in X$ be a realization of the Radon-Nikodym derivative, then the
 pullback  has the form
\[K(\theta_{1},\theta_{2}):=\int{\log{\frac{p(x|\theta_{1})}{p(x|\theta_{2})}}dP_{\theta_{1}}(x)}\, .\]
From this we define a Kullback-Leibler neighborhood of a point $\theta \in \Theta$ by
\[ K_{\epsilon}(\theta):=\bigl\{\theta' \in \Theta: K(\theta,\theta')\leq \epsilon\bigr\}\,.\]
Let us define the set of  priors $\mathcal{K}(\theta)\subset \mathcal{M}(\Theta)$ which
have  Kullback-Leibler {\em support at $\theta$}
   by
\begin{equation}
\label{def_Ktheta}
\mathcal{K}(\theta):=\Bigl\{\pi \in  \mathcal{M}(\Theta): \pi\bigl( K_{\epsilon}(\theta)\bigr) >0,\quad \epsilon >0 \Bigr\}\,,
\end{equation}
which implicitly requires that $ K_{\epsilon}(\theta)$ be measurable\footnote{
Note  the change from the standard definition $K_{\epsilon}(\mu)=\{\mu': K(\mu,\mu')<  \epsilon\}$
to ours $K_{\epsilon}(\mu)=\{\mu': K(\mu,\mu') \leq \epsilon\}$
does not affect which measures
have Kullback-Leibler support, but is more convenient since then
$K_{\epsilon}(\mu)$ is closed, simplifying the  proof  that $K_{\epsilon}(\theta)$ is measurable.}
 for all $\epsilon >0$.
Also let $\mathcal{K}\subset \mathcal{M}(\Theta)$ denote those measures with
{\em global} Kullback-Leibler support, that is,
\begin{equation*}
\label{def_Kglobal}
\mathcal{K}:=\cap_{\theta \in \Theta}{\mathcal{K}(\theta)}
\end{equation*}
 is the set of priors
which have Kullback-Leibler support at all $\theta$, and
let $\mathcal{K}^{ae} \supset \mathcal{K}$, defined by
\begin{equation}
\label{def_Kaglobal}
\mathcal{K}^{ae}:=\Biggl\{\pi \in \mathcal{M}(\Theta): \pi\Bigl\{\theta \in \Theta:\pi\bigl(K_{\epsilon}(\theta)\bigr)>0,\,  \epsilon> 0\Bigr\}=1 \Biggr\}\, ,
\end{equation}
denote the set of priors with
{\em almost global} Kullback-Leibler support.

Let us address the measurability of the Kullback-Leibler
neighborhoods $ K_{\epsilon}(\theta) \subset \Theta,\, \epsilon >0$. For the nonparametric case,
Barron, Schervish and Wasserman \cite[Lem.~11]{BarronEtAl:1999}
demonstrate  that
the Kullback-Leibler neighborhoods $K_{\epsilon}(P_{\theta^{*}}) \subset \mathcal{M}(X)$ are measurable with respect to the strong topology
restricted  to the subspace of measures which are absolutely continuous with respect to a common $\s$-finite reference measure.
For the parametric case,
Dupuis and Ellis \cite[Lem.~1.4.3]{DupuisEllis}   assert that on a Polish space  that
 K is lower semicontinuous in both arguments. Since the subset embedding
$:X \rightarrow X'$  of a subset $X$ of a  metric space $X'$ is isometric,  when $X$ is a Borel subset
of a separable metric space $X'$, it can be shown  that the induced
pushforward map $i_{*}:\mathcal{M}(X) \rightarrow \mathcal{M}(X')$ is isometric in the Prokhorov metrics, in particular it is continuous.
 Since the composition of a continuous and a lower semicontinuous function is lower semicontinuous, it follows from Dupuis and Ellis \cite[Lem.~1.4.3]{DupuisEllis}  that  on any
realization of a standard Borel space that the Kullback-Leibler divergence is lower semicontinuous in each of its arguments
separately, in particular,  fixing the first, it is lower semicontinuous. Therefore
$K_{\epsilon}(P_{\theta^{*}}) \subset \mathcal{M}(X)$ is closed, and therefore measurable for $\epsilon >0$.
   Consequently, when
$P$  is measurable, it follows that
$K_{\epsilon}(\theta^{*})\subset \Theta$ is measurable for $\epsilon >0$.

The following corollary to Schwartz' Theorem, and its implications in Proposition \ref{prop_schwartz}, gives
us the form of consistency that we will use in the robustness analysis.
Since the $\s$-algebra
of a Borel subset of a Polish space is countably generated,
 Doob's Theorem in Dellacherie and Meyer \cite[Thm. V.58]{dellacherie1982probabilities} and the measurability
of the dominated model $P$  implies that
 a
  family $p(\theta), \theta \in \Theta$ of  densities can be chosen to be
$\mathcal{B}(X) \times \mathcal{B}(\Theta)$ measurable, so that
this assumption of Schwartz' Theorem \cite[Thm.~6.1]{Schwartz:1965} is satisfied.
We note that Dellacherie and Meyer emphasize that the countably generated condition is indispensable
for  Doob's Theorem  to apply.
  Note
the assumption that the map $P:\Theta \rightarrow P(\Theta)$ be open.
\begin{cor}[Schwartz]
\label{cor_schwartz}
Let $X$  and $\Theta$ be Borel subsets of Polish metric spaces and equip $\mathcal{M}(X)$ and  $\mathcal{M}(\Theta)$
 with the Prokhorov metrics. Consider  an injective measurable dominated model
$P:\Theta \rightarrow \mathcal{M}(X)$
such that $P:\Theta \rightarrow P(\Theta)$  is open.
Then for every
   $\pi \in \mathcal{M}(\Theta)$ with Kullback-Leibler support at $\theta^{*} \in \Theta$ and
for every  measurable neighborhood $U$ of $\theta^{*}$,
we have
\[ \pi_{x^{n}}(U) \rightarrow 1\quad  n\rightarrow \infty,\quad   a.e.~P^{\infty}_{\theta^{*}}\, .\]
\end{cor}

\section{Main Results}
\label{sec_main}
Now that we have defined \emph{qualitative robustness} for Bayesian inference and presented the
consistency conditions of Schwartz' Corollary \ref{cor_schwartz},
we are now prepared for our main results.   Indeed,  the brittleness results of
 \cite{OSS:2013,OwhadiScovel:2013,owhadi_shatters} and the
non qualitative robustness results of Cuevas \cite[Thm.~7]{Cuevas} suggest that we may obtain non
 qualitative robustness
according to
Definition \ref{def_qr_kyfan}
by fixing the prior and varying the data generating distribution.   However,  according to Berk \cite{Berk:1966},
in the misspecified case, although ``there need be no convergence (in any sense)'',
 in the limit the posterior becomes confined to a carrier  set consisting of those points which are closest in terms of the Kullback-Leibler divergence. Consequently, it appears possible that a generalization of
the results of Hampel \cite[Lem.~3]{Hampel} and  Cuevas \cite[Thm.~1]{Cuevas} which  allows such a  set-valued
notion of consistency may be sufficient. Certainly it will require the more sophisticated notions of the
continuity, or semi-continuity, of the Kullback-Leibler set-valued information projection and its dependence
on the geometry of the model class $P(\Theta) \subset \mathcal{M}(X)$.
Although this path will certainly be instructive and appears feasible,   we instead find it simpler
  to obtain non qualitative robustness
by
 fixing the data generating distribution to be in the model class and varying the prior.
In particular, we show that the inference is not robust according to Definition \ref{def_qr_kyfan}
 when the metric $\alpha_{\mu}$ is the
the Ky Fan metric and the metric $\mathcal{M}(\Theta)$ is any that is weaker than the total variation metric.
 It is important to note that
 these results do not require any misspecification. Moreover, it appears that Bayesian Inference's dependence
on both the data generating distribution and the prior leads to two complementary mechanisms generating
non qualitative robustness; whereas Cuevas' result \cite[Thm.~7]{Cuevas} utilizes consistency and the discontinuity
of the infinite sample limit, this other component utilizes the non-robustness
of consistency, namely that the set of consistency priors, those with Kullback-Leibler support at the data generating distribution, is not robust.

Now let us return to our main results.
For $\theta \in \Theta$,  let us recall from \eqref{def_Ktheta}  the set of priors
$\mathcal{K}(\theta) \subset \mathcal{M}(\Theta)$  with Kullback-Leibler support at $\theta$
and,  for  $\rho>0$, define a total variation uniformity
 $\Pi_{\rho}(\theta)\subset \mathcal{M}(\Theta) \times \mathcal{M}(\Theta) $
by
\begin{equation*}
\label{def_Pi}
 \Pi_{\rho}(\theta):=\{(\pi,\acute{\pi}) \in  \mathcal{M}(\Theta) \times \mathcal{M}(\Theta):
\pi \in  \mathcal{K}(\theta),\,\, d_{tv}(\pi,\acute{\pi}) < \rho\}\,
\end{equation*}
of prior pairs where the first component has Kullback-Leibler support at $\theta$ and
the second component is within $\rho$ of the first in the total variation metric.
For $\theta \in \Theta$, we define an admissible set of prior-data generating distribution pairs
 $\mathcal{Z}_{\rho}(\theta) \subset \bigl(\mathcal{M}(\Theta) \times \mathcal{M}(X)\bigr)^{2}   $
by
\begin{equation}
\label{def_Z}
\mathcal{Z}_{\rho}(\theta):= \Pi_{\rho}(\theta)\times P_{\theta}\times
P_{\theta} \, ,
\end{equation}
 using the identification  of
$\bigl(\mathcal{M}(\Theta) \times \mathcal{M}(X)\bigr)^{2}$ with $\mathcal{M}(\Theta)^{2}
\times  \mathcal{M}(X)^{2}$.

Our Main Theorem shows,  under the conditions of Schwartz' Corollary,  that the Bayesian inference is not robust
under the assumption that the prior has Kullback-Leibler support at the parameter value generating the data.
This result, along with those that follow,  supports  Cuevas' \cite{Cuevas2}  statement that ``his results
suggest the possibility of proving the instability (i.e.~the lack of qualitative robustness) for a wide class of usual Bayesian models."
\begin{thm}
\label{thm_main}
Consider Definition \ref{def_qr_kyfan} with the total variation metric $d_{\mathcal{M}(\Theta)}:=d_{tv}$ on
$\mathcal{M}(\Theta)$ and the Ky Fan metric $\alpha_{\mu}$
on  the space of $\mathcal{M}(\Theta)$-valued random variables with domain
the  probability space $(X^{\infty},\mu^{\infty})$.  Given the conditions of
Schwartz' Corollary \ref{cor_schwartz},
for all $\theta \in \Theta$
   the Bayesian inference is not qualitatively robust with respect to
 $\mathcal{Z}_{\rho}(\theta)$ for  all
  $\rho >0$.
\end{thm}
\begin{rmk}
Actually the proof shows more; let $D$ denote the diameter of $\Theta$, then for
  $\epsilon <\min{(\frac{D}{2},1)}$, there does not exist a $\delta >0$ such that
 robustness is satisfied.  Since $\min{(\frac{D}{2},1)}$  is large, either half
the diameter of the space or larger than $1$,  we say  the inference is
{\em brittle}.
\end{rmk}

Theorem \ref{thm_main} does not assert that the Bayesian inference is not robust at any specified prior, only that
it is not robust under the assumption
that  the prior has Kullback-Leibler support at the parameter value generating the data. To establish non-robustness
at specific priors we include variation in the data-generating distribution in the model class as follows.
Let $\Delta_{P} \subset \mathcal{M}(X)\times \mathcal{M}(X)$, defined by
\[\Delta_{P}=\{(P_{\theta},P_{\theta}), \theta \in \Theta\}\,,\]
denote the fact that we allow the data generating distribution to vary throughout the model class but do not allow any perturbations to it.
Then,
for $\pi \in \mathcal{M}(\Theta)$, define the  admissible set
 $\mathcal{Z}_{\rho}(\pi) \subset \bigl(\mathcal{M}(\Theta) \times \mathcal{M}(X)\bigr)^{2}   $
by
\begin{equation*}
\label{def_Zbis}
\mathcal{Z}_{\rho}(\pi):= \pi \times B^{tv}_{\rho}(\pi) \times \Delta_{P}
\, ,
\end{equation*}
where $B^{tv}_{\rho}(\pi)$ is the open ball in the total variation metric.

Since the following theorem is a corollary to the theorem after it, Theorem \ref{thm_main3}, we do not include its proof. However, we state it here because it is the more fundamental result.
\begin{thm}
\label{thm_main2}
Consider the sitiuation of Theorem \ref{thm_main}
 with  $\Theta$ not totally bounded. Then if the prior
$\pi$ has Kullback-Leibler support for all $\theta \in \Theta$,
   the Bayesian inference is not qualitatively robust with respect to
 $\mathcal{Z}_{\rho}(\pi)$ for  all
  $\rho >0$.
\end{thm}
 Since a metric space is totally bounded
if and only if its completion is compact, when $\Theta$ is totally bounded, we  assume
that it is a Borel subset of a compact metric space. In this case, although Theorem \ref{thm_main2} does not apply,
 utilizing  the covering number and packing number inequalities of
 Kolmogorov and Tikhomirov \cite{KolmogorovTikhomirov},
we can provide a  natural {\em quantification of qualitative robustness}.
To that end, we define covering  and  packing numbers. For a finite subset $\Theta'\subset \Theta$, the finite collection of open balls
$\{B_{\epsilon}(\theta), \theta \in \Theta'\}$
is said to constitute a covering of $\Theta$ if
$\Theta \subset \cup_{\theta \in \Theta'}{B_{\epsilon}(\theta)}$.  For a finite
set $\Theta'$ we denote its size by $|\Theta'|$. The covering numbers are defined by
\[\mathcal{N}_{\epsilon}(\Theta)= \min \Bigl\{|\Theta'|: \Theta \subset \cup_{\theta \in \Theta'}{B_{\epsilon}(\theta)} \Bigr\}\, , \]
that is, $\mathcal{N}_{\epsilon}(\Theta) $ is the smallest number of open
balls of radius $\epsilon$ centered on points in $\Theta$ which covers $\Theta$.
On the other hand,  a set of points $\Theta' \subset \Theta$ is said to constitute
an $\epsilon$-packing  if $d(\theta_{1},\theta_{2})\geq \epsilon, \, \theta_{1}\neq \theta_{2} \in \Theta'$.
The packing numbers are then defined by
\[\mathcal{M}_{\epsilon}(\Theta):=\max \Bigl\{|\Theta'|: \Theta' \,\, \text{ is an $\epsilon$-packing
of}\,\, \Theta\,  \Bigr\}\, .\]
Since the Kolmogorov and Tikhomirov \cite[Thm.~IV]{KolmogorovTikhomirov}  inequalities
\begin{equation}
\label{eq_kolm}
\mathcal{M}_{2\epsilon}(\Theta) \leq \mathcal{N}_{\epsilon}(\Theta) \leq \mathcal{M}_{\epsilon}(\Theta)\,
\end{equation}
 are valid in the {\em not} totally bounded case,
if we allow values of $\infty$,
the following theorem has Theorem \ref{thm_main2} as its corollary.
\begin{thm}
\label{thm_main3}
Given the conditions of Theorem \ref{thm_main2} with $\Theta$  totally  bounded.
 If  the Bayesian inference
  is
qualitatively robust with respect to $\mathcal{Z}_{\rho}(\pi)$ for some $\rho >0$,  then given $\epsilon_{2} >0$,   we must have
\[ \delta_{1} < \min{\Bigl(\frac{1}{\mathcal{N}_{2\epsilon_{2}}(\Theta)}, \rho \Bigr)} \, .\]
\end{thm}

\section{Mechanisms generating non-robustness}\label{subsec:mechanism}
For the clarity of the paper, in this subsection, we illustrate some of the mechanisms generating
non qualitative robustness in Bayesian inference, which complement the mechanism discovered by
Hampel \cite[Lem.~3]{Hampel}
 and Cuevas \cite[Thm.~1]{Cuevas}.  These mechanisms do not utilize  misspecification. Those which do
 are discussed
in Subsection \ref{subsecmiss}.
The core mechanism is derived from the nature of both the assumptions and assertions of results supporting consistency.
More precisely,
Corollary \ref{cor_schwartz} states that if the data generating distribution is $\mu=P(\theta^*)$ and if the prior $\pi$ attributes positive mass to every Kullback-Leibler neighborhood of $\theta^*\in \Theta$, then the posterior
 distribution converges towards $\delta_{\theta^*}$ as $n\rightarrow \infty$.
	\begin{figure}[tp]
\begin{center}
			\includegraphics[width=0.8\textwidth]{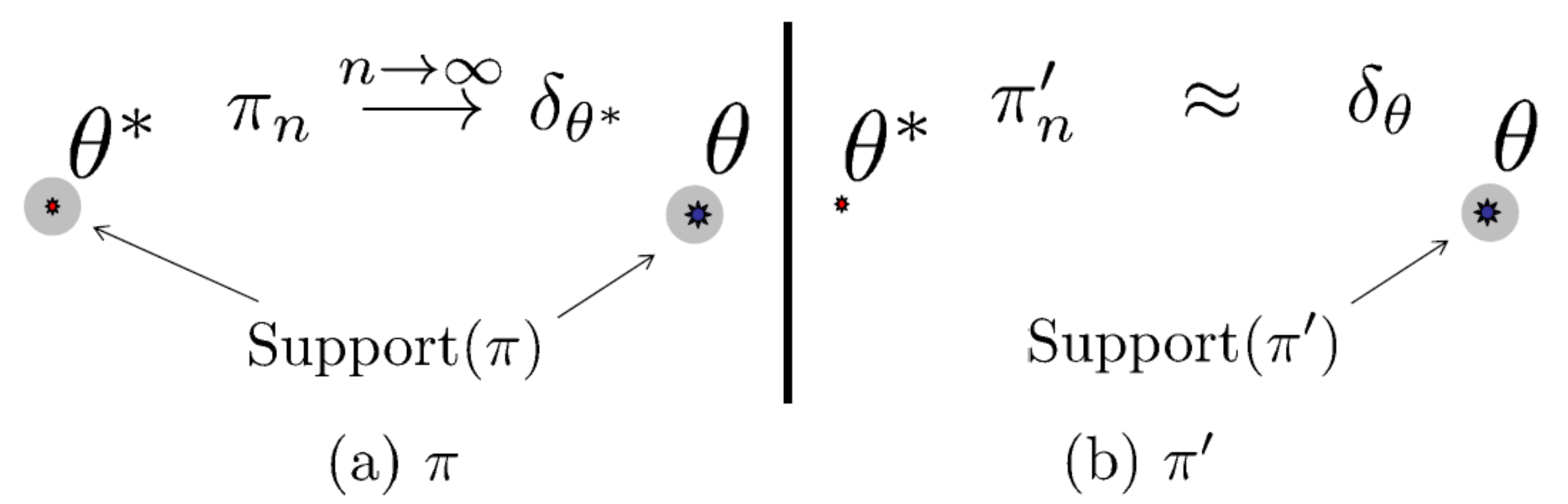}
		\caption{The data generating distribution is $P(\theta^*)$.  $\pi'$  has most of its mass around $\theta$. $\pi$ is an arbitrarily small perturbation of $\pi'$ so that $\pi$ has Kullback-Leibler support at $\theta^*$. Corollary \ref{cor_schwartz} implies that $\pi_n$ converges towards $\delta_{\theta^*}$ while $\pi_n'$ remains close to $\delta_\theta$.}
		\label{fig:bigfig1}
\end{center}
	\end{figure}
The assumption that $\pi$ attributes positive mass to every Kullback-Leibler neighborhood of $\theta^*\in \Theta$ does not require
    $\pi$ to place a significant amount of mass around $\theta^*$, but instead
can be satisfied with an arbitrarily small amount. Therefore, if, as in Figure \ref{fig:bigfig1}, $\pi$ is a prior distribution with support centered around $\theta\not=\theta^*$, but with a very small amount of mass
about $\theta^*$, so that it satisfies the  assumptions of Corollary \ref{cor_schwartz} at $\theta^*$, then $\pi$ can be slightly perturbed into a $\pi'$
  with support also centered around $\theta\not=\theta^*$, but with no  mass about $\theta^*$.
   In this situation, although $\pi$ and $\pi'$ can be made arbitrarily close in total variation distance, the posterior distribution of $\pi$ converges towards $\delta_{\theta^*}$ as $n\rightarrow \infty$, whereas that of $\pi'$ remains  close to $\delta_\theta$.
	\begin{figure}[tp]
\begin{center}
	\subfigure[Density of $\pi$]{
			\includegraphics[width=0.3\textwidth]{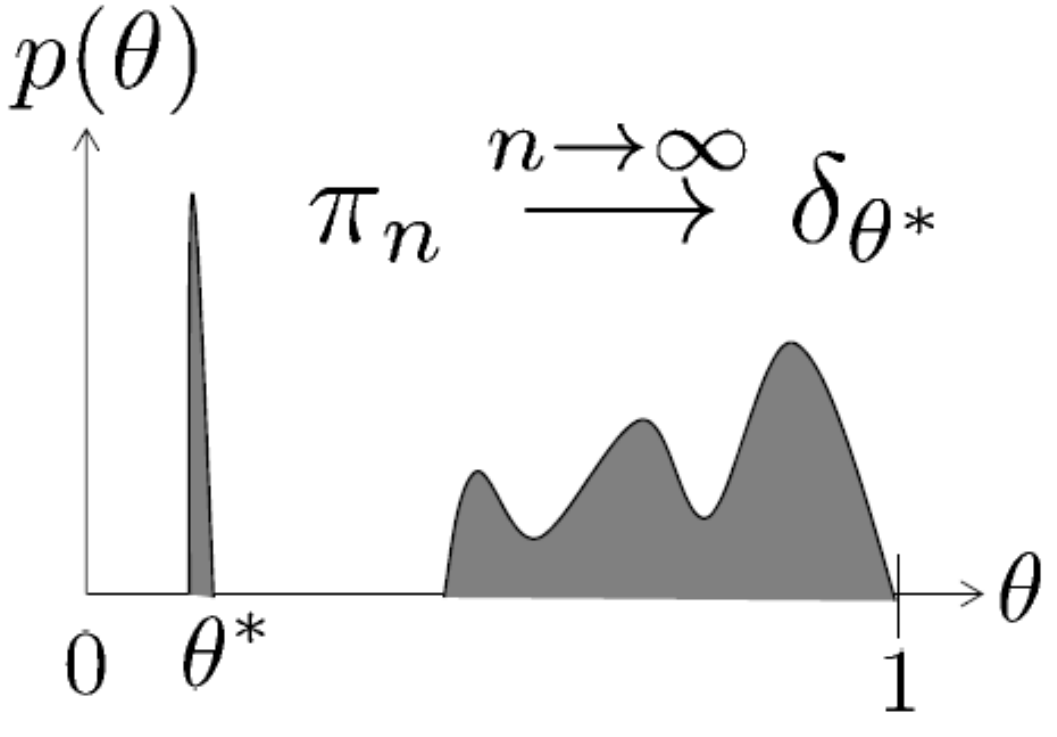}
		}
		\subfigure[Density of $\pi$]{
			\includegraphics[width=0.3\textwidth]{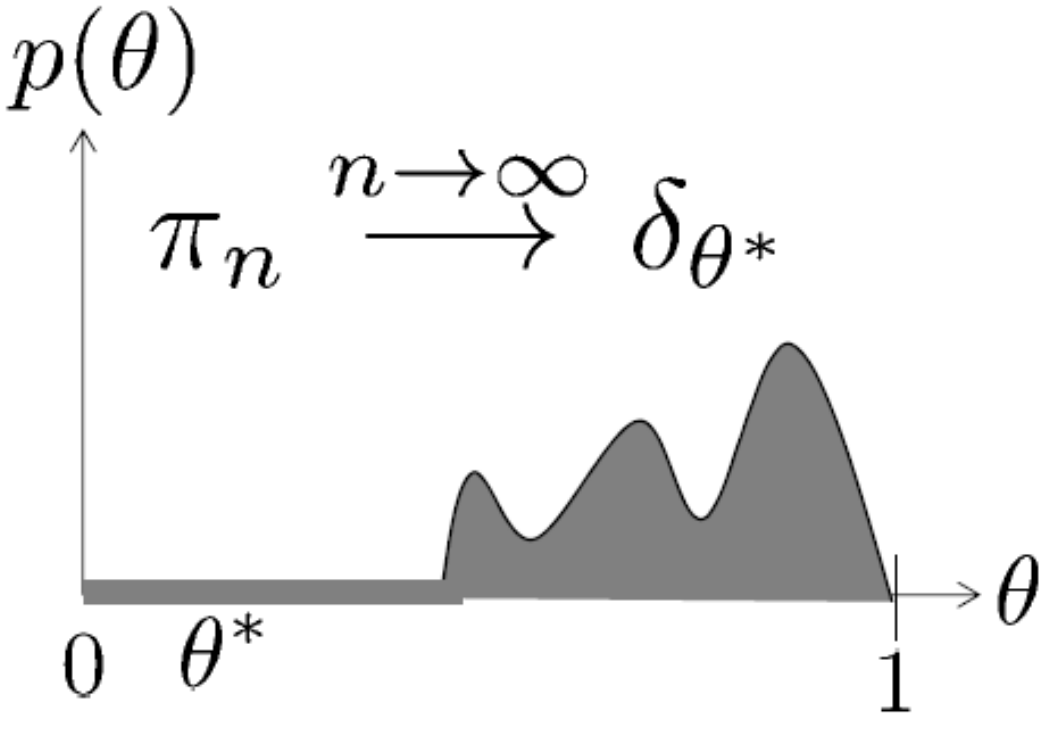}
		}
		\subfigure[Density of $\pi'$]{
			\includegraphics[width=0.3\textwidth]{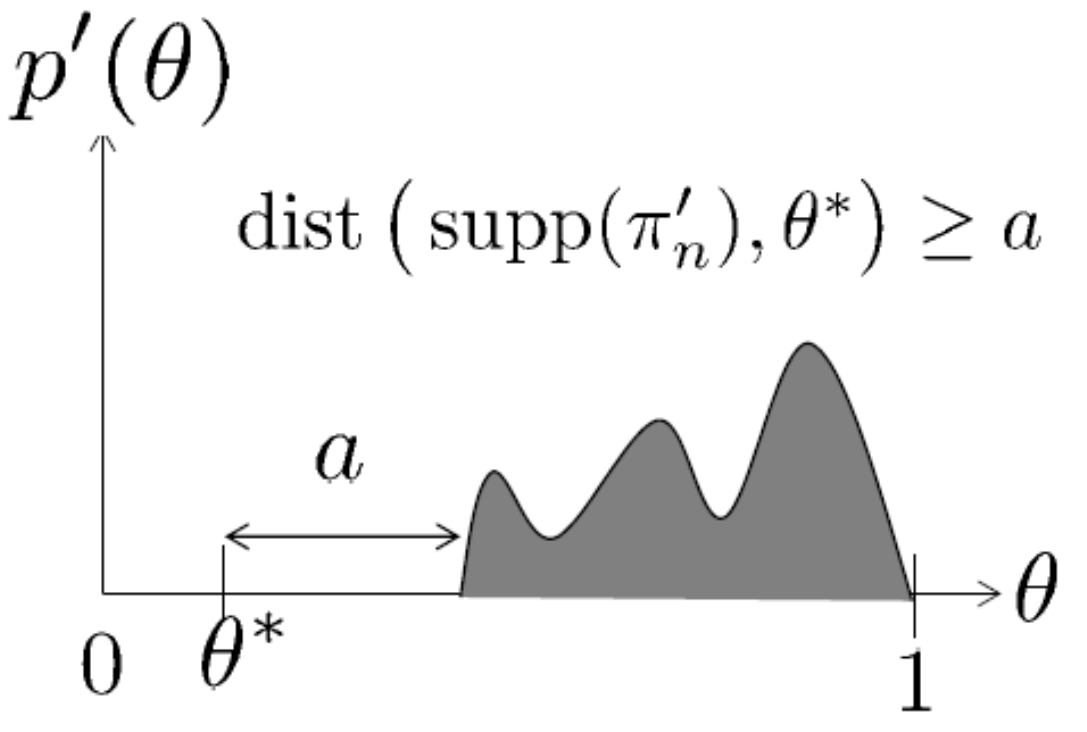}
		}
		\caption{The data generating distribution is $P(\theta^*)$. The probability density functions of $\pi$ and $\pi'$ are $p$ and $p'$ with respect to the uniform distribution on $[0,1]$. $\pi$  is an arbitrarily  small perturbation of $\pi'$ in total variation.
 $\pi_n$ converges towards $\delta_{\theta^*}$ while the distance between the support of $\pi_n'$ and $\theta^*$ remains bounded from below by $a>0$.}
		\label{fig:bigfig2}
\end{center}
	\end{figure}
Figure \ref{fig:bigfig2}
gives an illustration of the same phenomenon when the parameter space $\Theta$ is the interval $[0,1]$ and the probability density functions of $\pi$ and $\pi'$ are $p$ and $p'$ with respect to the uniform measure.

Note that the mechanism illustrated in Figures \ref{fig:bigfig1} and \ref{fig:bigfig2} does not generate non
qualitative robustness at {\em  all priors}  but instead  for the  full class of {\em consistency priors}, defined by the assumption of having positive mass on every Kullback-Leibler neighborhood of $\theta^*$.
One may wonder whether this non qualitative robustness can be avoided by selecting the prior $\pi$ to satisfy Cromwell's rule (that is, the assumption that $\pi$ gives strictly positive mass to every nontrivial open subset of the parameter space $\Theta$). Theorem \ref{thm_main2} shows that this is not the case if the parameter space $\Theta$ is not totally bounded. For example, when $\Theta=\R$, for all $\delta>0$ one can find $\theta\in \R$ such that the mass that $\pi$ places on the ball of center $\theta$ and radius one is smaller than $\delta$, and by displacing this small amount of mass one obtains a perturbed prior $\pi'$ whose posterior distribution remains asymptotically bounded away from that of $\pi$ when the data-generating distribution is $P(\theta)$. Similarly if $\Theta$ is
 totally bounded then Theorem \ref{thm_main3} places an upper bound on the size of the perturbation
 of the prior $\pi$ that would be required as a function of the covering complexity of $\Theta$.
Note that these observations suggest that a maximally qualitatively robust prior should place as much mass as possible near all possible candidates $\theta$ for the parameter $\theta^*$ of the data generating distribution, thereby reinforcing the notion that a maximally robust prior should have its mass spread as uniformly as possible over the parameter space. 

	\begin{figure}[tp]
\begin{center}
	\subfigure[$\pi$]{
			\includegraphics[width=0.4\textwidth]{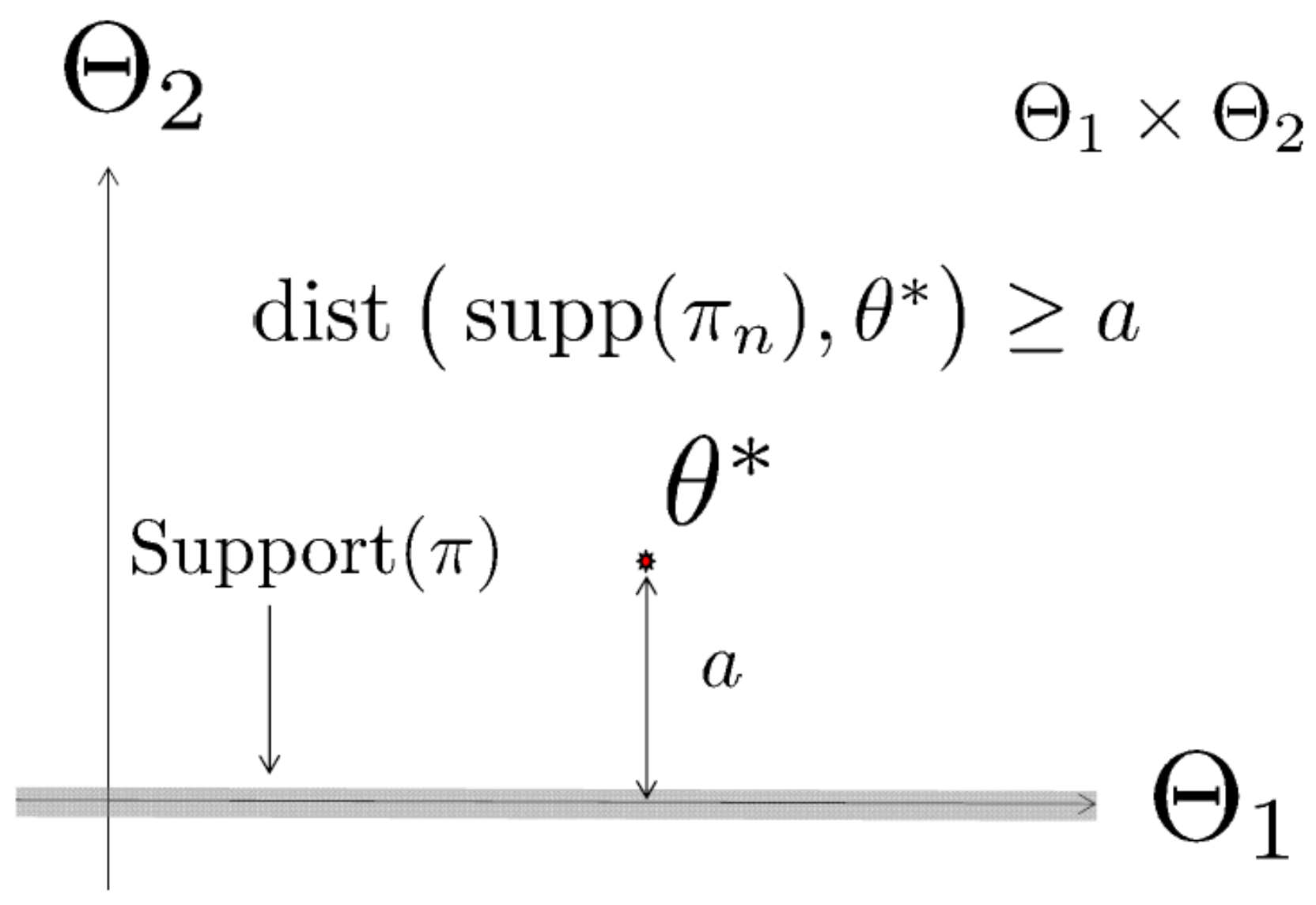}
		}
		\subfigure[$\pi'$]{
			\includegraphics[width=0.4\textwidth]{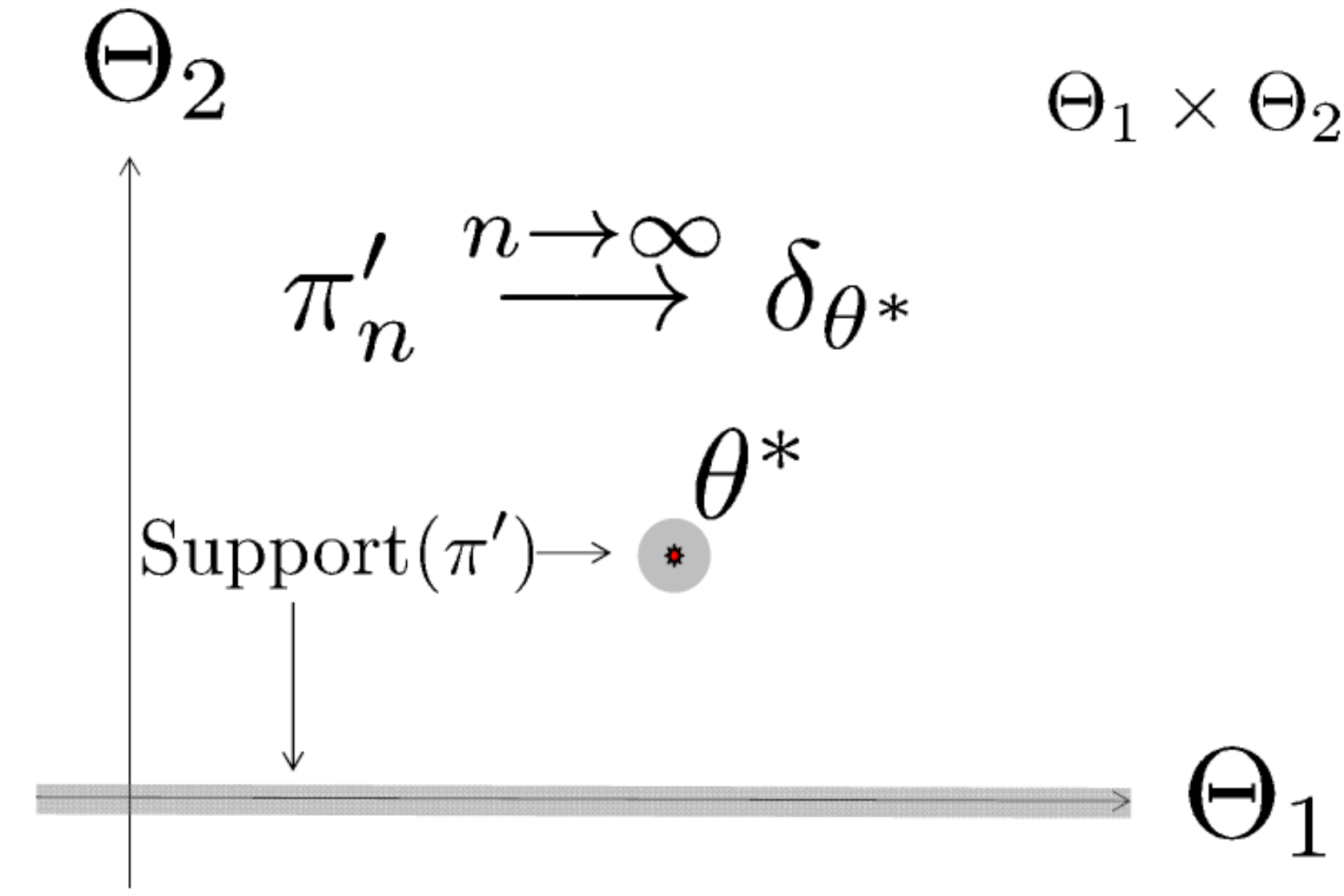}
		}
		\caption{Non-robustness caused by misspecification.
The parameter space of the model is $\Theta_1$. We assume that the model $P \,:\, \Theta_1 \rightarrow \mathcal{M}(X)$ is the restriction of an
injective model $\bar{P}\,:\, \Theta_1\times \Theta_2 \rightarrow \mathcal{M}(X)$ to $\Theta_1\times \{\theta_2=0\}$. The data generating distribution is $\bar{P}(\theta^*)$ where $\theta^{*}:=(\theta_{1}^{*},\theta_{2}^{*})$, with $\theta_{2}^{*}\neq 0$, so that the model $P$ is misspecified. $\pi$ satisfies Cromwell's rule. $\pi'$ is an arbitrarily small perturbation of $\pi$ having Kullback-Leibler support at $\theta^*$. Corollary \ref{cor_schwartz} implies that $\pi_n'$ converges towards $\delta_{\theta^*}$ while the distance between the support of $\pi_n$ and $\theta^*$ remains bounded from below by $a>0$.}
		\label{fig:bigfig3}
\end{center}
	\end{figure}
	\begin{figure}[tp]
\begin{center}
	\subfigure[$P_*\pi$]{
			\includegraphics[width=0.35\textwidth]{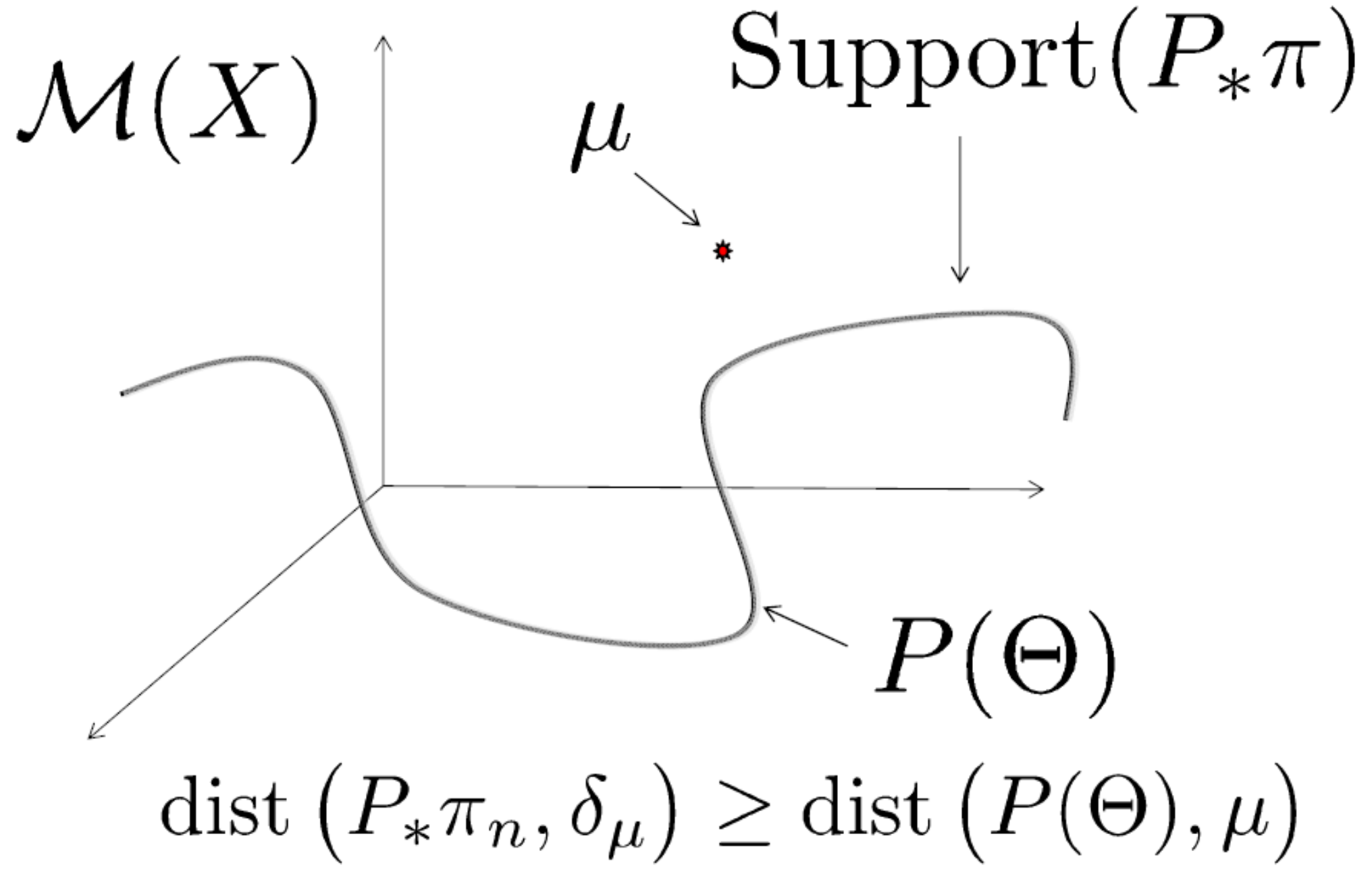}
		}
		\subfigure[$\nu$]{
			\includegraphics[width=0.35\textwidth]{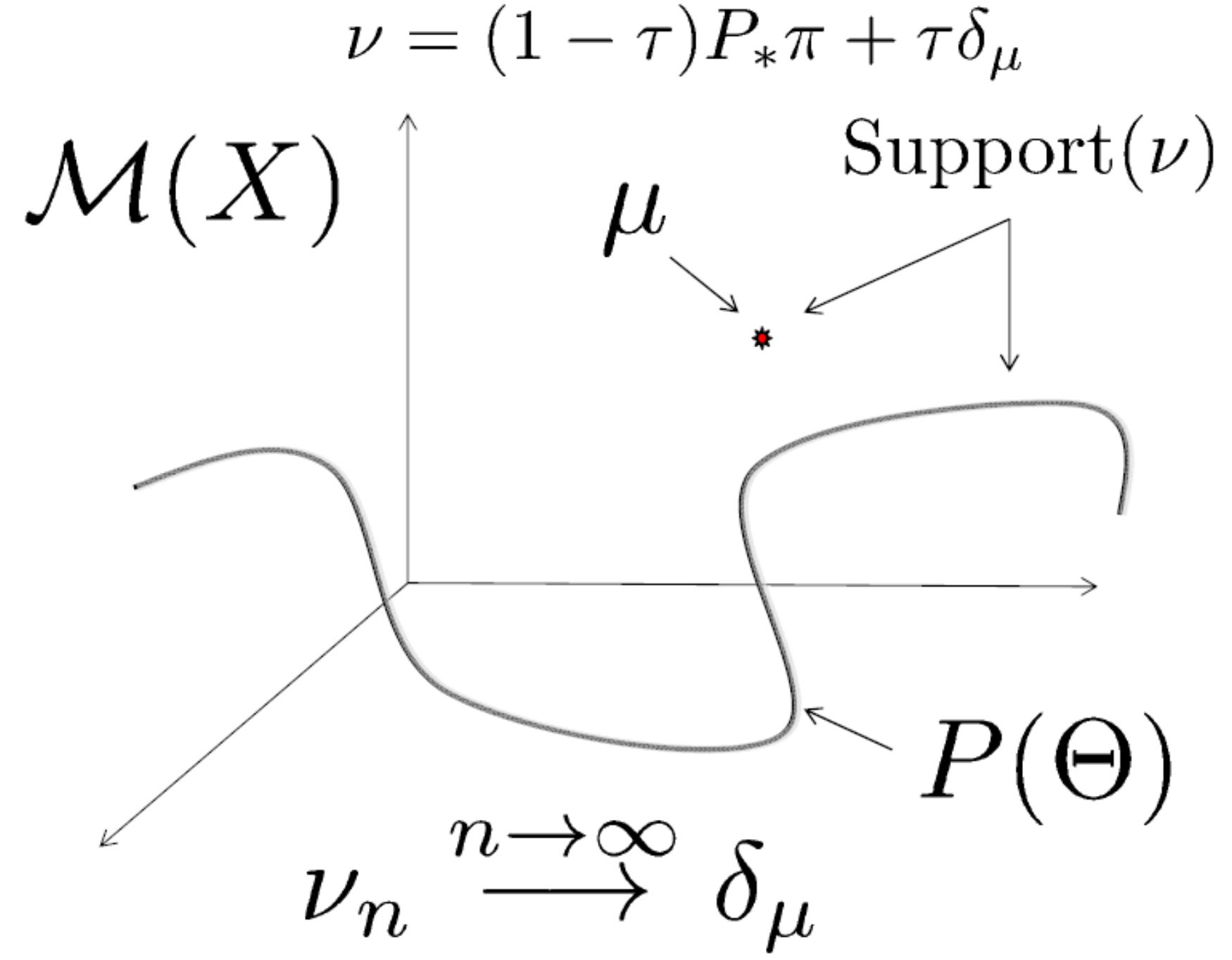}
		}
		\caption{Non-robustness caused by misspecification. The parameter space of the model is $\Theta$. The data generating distribution is $\mu \not\in P(\Theta)$, so that the model is misspecified.
$\nu$ is an arbitrarily small perturbation of $P_*\pi$ in total variation distance having non-zero mass on $\mu$. Note that if a small mass on $\mu$ is sufficient to ensure the consistency of $\nu$ then such a mechanism also implies the non-robustness of the model with respect to misspecification since in such a case, $\nu_n$, the posterior distribution of $\nu$ would converge towards $\delta_{\mu}$ whereas the distance between the support of $P_* \pi_n$ and $\mu$ remains bounded from below by the distance from the model to the data generating distribution.}\label{fig:bigfig4}
\end{center}
	\end{figure}

\subsection{Robustness under misspecification}\label{subsecmiss}
Although the main results of Section \ref{sec_main}
 do not utilize any  model misspecification, the brittleness results of
\cite{OSS:2013} suggest that misspecification should also generate non qualitative robustness.
Indeed, although, one may find a prior that is both consistent and \emph{qualitatively robust} when $\Theta$ is totally bounded and the model is well-specified, we now show how  extensions of the mechanism illustrated in Figures \ref{fig:bigfig1} and \ref{fig:bigfig2} suggest that misspecification implies non \emph{qualitative robustness}.
Consider the example illustrated in Figure \ref{fig:bigfig3}. In this example the model $P$ is the restriction of a well specified larger model $\bar{P} \,:\, \Theta_1 \times \Theta_2 \rightarrow \mathcal{M}(X)$ to $\theta_2=0$. Assume that the data generating distribution is $\bar{P}(\theta_1^*,\theta_2^*)$ where $\theta_2^*\not=0$, so that the restricted model $P$ is misspecified. Let $\pi$ be any prior distribution on $\Theta_1 \times \{\theta_2=0\}$.  Although $\pi$ may satisfy Cromwell's rule the mechanisms presented in this paper suggest that is not \emph{qualitatively robust} with respect to  perturbed priors having support on $\Theta_1 \times \Theta_2$. Indeed, let $\pi'$ be an arbitrarily small perturbation of $\pi$ obtained by removing some mass from the support of $\pi$ and adding that mass around $\theta^*$. Note that
 $\pi'$ can be chosen arbitrarily close to $\pi$ while satisfying the local consistency assumption of
Corollary \ref{cor_schwartz}, which implies that the posterior distributions of $\pi'$ concentrate on $\theta^*$ while the posterior distributions of $\pi$ remain supported on $\Theta_1 \times \{\theta_2=0\}$.
Note that if $\bar{P}$ is interpreted as an extension of the model $P$, then this mechanism suggests that we can
 establish conditions under which Bayesian inference is not \emph{qualitatively robust under model extension}.

Figure \ref{fig:bigfig4} represents a
 non-parametric generalization of the  mechanism of Figure \ref{fig:bigfig3}.  Assume that the data generating distribution is $\mu \not \in P(\Theta)$, so that the model is misspecified.
 Let $\pi \in \mathcal{M}(\Theta)$  be an arbitrary prior distribution and $P_{*}\pi \in \mathcal{M}^{2}(X)$ its corresponding non-parametric prior.
 By removing an arbitrarily small amount of mass from $P_* \pi$ and placing it on $\mu$ one  obtains an arbitrarily close prior distribution $\nu$ that is consistent with respect to the data generating distribution $\mu$. Therefore although  $P_* \pi$ and $\nu$ may be made arbitrarily close, their posterior distributions would remain asymptotically separated by a distance corresponding to the degree of misspecification of the model (the distance from $\mu$ to $P(\Theta)$).

\section{Proofs}
\subsection{Proof of Corollary \ref{cor_schwartz}}
We seek to apply Schwartz' theorem \cite[Thm.~6.1]{Schwartz:1965}.
Since  $\Theta$ and $X$ are separable metric spaces,
their Borel $\s$-algebras are countably generated, so
Doob's Theorem \cite[Thm. V.58]{dellacherie1982probabilities}
 and the measurability
of the dominated model $P$ implies  that
 a
  family  of  densities can be chosen to be
$\mathcal{B}(X) \times \mathcal{B}(\Theta)$ measurable, thus satisfying
this requirement of  \cite[Thm.~6.1]{Schwartz:1965}.
Since $U$ is a neighborhood it follows that it contains an open neighborhood $O$ of $\theta$.
Since $O$ is open and $P:\Theta \rightarrow P(\Theta)$ is open, it follows
that $P(O)$ is open in $P(\Theta)$, and therefore there is an open set
$V_{*}\subset \mathcal{M}(X)$ such that $V_{*}\cap P(\Theta)=P(O)$. Moreover, $V_{*}$ is an open neighborhood
of $P_{\theta^{*}}$.
Since $X$ is a separable metric space, it follows that
 $d_{PrX}$ metrizes the weak topology, and since $V_{*}$ is open, it is well known (see e.g.~\cite{BarronEtAl:1999,Wasserman_asymptotic,Ghosal_issues})
 that there exists a uniformly consistent test
of $P_{\theta^{*}}$ against $V^{c}_{*}$, see Schwartz \cite{Schwartz:1965} for the definition of uniformly consistent test.
It follows trivially that  there exists a uniformly consistent test of
$P_{\theta^{*}}$ against $V^{c}_{*}\cap P(\Theta)$.
 Moreover,  since
$P$ is injective it follows that $O^{c}=P^{-1}(V^{c}_{*})$. Therefore,
there exists a uniformly consistent test
of $P_{\theta^{*}}$ against $ V^{c}_{*}\cap P(\Theta)=\{P_{\theta}: \theta \in O^{c}\}
$.

Since $V_{*}$ is open, it also follows that there is a Prokhorov metric ball $B_{s}(P_{\theta^{*}})$  of radius $s>0$
about $P_{\theta^{*}}$
such that $B_{s}(P_{\theta^{*}}) \subset V_{*}$.
 Now consider the
Kullback-Leibler ball
$K_{\tau}(P_{\theta^{*}})$ for $\tau <\frac{s^{2}}{2}$.  It follows from   Csiszar, Kemperman and Kullback's \cite{Csiszar_divergence}
  improvement
$K \geq \frac{1}{2}d_{tv}^{2} $
of Pinsker's inequality and the inequality $d_{tv} \geq d_{PrX}$, that $K_{\tau}(P_{\theta^{*}}) \subset B_{s}(P_{\theta^{*}})$.
Since then $K_{\tau}(P_{\theta^{*}}) \subset B_{s}(P_{\theta^{*}})\subset V_{*}$ it follows that
\[P^{-1}\bigl(K_{\tau}(P_{\theta^{*}})\bigr)  \subset P^{-1}(V_{*})= O\, .\]

Consider now the Kullback-Leibler neighborhood $W_{\tau}(\theta^{*})\subset \Theta$ of
 $\theta^{*}$ defined by
pulling $K_{\tau}(P_{\theta^{*}})$ back to $\Theta$ by the model $P$:
\[W_{\tau}(\theta^{*}):=P^{-1}\bigl(K_{\tau}(P_{\theta^{*}})\bigr)\, .\]
Then the previous inequality states that
 \[W_{\tau}(\theta^{*}) \subset O \, .\]
Since the Kullback-Leibler neighborhoods are measurable in the weak topology
and $P$ is assumed measurable, it follows that $W_{\tau}(\theta^{*})$ is measurable.

Therefore, $O$ and $W_{\tau}(\theta^{*})$ satisfy the assumptions of the sets $V$ and
and $W$ in
 \cite[Thm.~6.1]{Schwartz:1965}.
 Consequently, since by assumption, the prior $\pi$ has Kullback-Leibler support, it follows that
we can apply  Schwartz' theorem \cite[Thm.~6.1]{Schwartz:1965} to obtain the assertion for
$O$ and since $U\supset O$ is measurable the assertion follows.

\subsection{Proof of Theorem \ref{thm_main}}
Let us prove the assertion for a weaker pseudometric $\acute{\alpha}_{\mu}\leq \alpha_{\mu}$
 derived from the Prokhorov metric
on $d_{Pr^{2}\Theta}$ on $\mathcal{M}^{2}(\Theta)$. Since it is weaker the assertion follows. To that end,
consider  $\mu \in \mathcal{M}(X)$. Then
for two random variables $Z,W:(X^{\infty},\mu^{\infty}) \rightarrow \mathcal{M}(\Theta)$ it follows that
$Z_{*}\mu^{\infty},W_{*}\mu^{\infty} \in \mathcal{M}^{2}(\Theta)$,  so we can define
a pseudometric  $\acute{\alpha}_{\mu}$  by
\begin{equation}
\label{eq_prokhovdef}
\acute{\alpha}_{\mu}(Z,W):=d_{Pr^{2}\Theta}(Z_{*}\mu^{\infty},W_{*}\mu^{\infty})\, .
\end{equation}
Since
 Dudley \cite[Thm.~11.3.5]{Dudley:2002} asserts that
\begin{equation}
\label{ineq_kf_pr}
d_{Pr^{2}\Theta}(Z_{*}\mu^{\infty},W_{*}\mu^{\infty}) \leq  \alpha_{\mu}(Z,W)\,,
\end{equation}
we conclude that
\begin{equation}
\label{eq_weaker}
\acute{\alpha}_{\mu}\leq \alpha_{\mu}\, .
\end{equation}

For fixed $\pi$ and $\mu$ and $n$,   the   $\mathcal{M}(\Theta)$-valued random variable
\[\pi_{n}:(X^{\infty},\mu^{\infty}) \rightarrow \mathcal{M}(\Theta)\, \]
defined by
$\pi_{n}(x^{\infty}):=\pi_{x^{n}}$ in  \eqref{eq_dekjjibib}   satisfies
 \begin{equation}
\label{eq_laws}
(\pi_{n})_{*}\mu^{\infty}=
\pi_{*}\mu^{n}
\end{equation}
where $\pi_{*}:\mathcal{M}(X^{n})\rightarrow \mathcal{M}^{2}(\Theta)$ is the pushforward operator corresponding to the
 map
$\bar{\pi}:X^{n} \rightarrow \mathcal{M}(\Theta)$ defined by
$\bar{\pi}(x^{n}):=\pi_{x^{n}}$  in \eqref{def_barpi}.
Consequently, we obtain
\begin{equation}
\label{eq_ivuyvuyv}
\acute{\alpha}_{\mu}(\pi_{n},\acute{\pi}_{n})=d_{Pr^{2}\Theta}(\pi_{*}\mu^{n},\acute{\pi}_{*}\mu^{n})\,
\end{equation}
for  $\pi, \acute{\pi} \in \mathcal{M}(\Theta)$  and $n$ fixed.
From the triangle inequality we then obtain
\begin{eqnarray*}
d_{Pr^{2}\Theta}(\pi_{*}\mu^{n}, \acute{\pi}_{*}\acute{\mu}^{n})& \leq &
d_{Pr^{2}\Theta}(\pi_{*}\mu^{n}, \acute{\pi}_{*}\mu^{n})+d_{Pr^{2}\Theta}(\acute{\pi}_{*}\mu^{n}, \acute{\pi}_{*}\acute{\mu}^{n})\\
&\leq& \alpha_{\mu}(\pi_{n},\acute{\pi}_{n}) +d_{Pr^{2}\Theta}(\acute{\pi}_{*}\mu^{n}, \acute{\pi}_{*}\acute{\mu}^{n})\, ,
\end{eqnarray*}
bounding the simple single term $d_{Pr^{2}\Theta}(\pi_{*}\mu^{n}, \acute{\pi}_{*}\acute{\mu}^{n})$ in terms
of the sum of two terms $\alpha_{\mu}(\pi_{n},\acute{\pi}_{n})$ and
$d_{Pr^{2}\Theta}(\acute{\pi}_{*}\mu^{n}, \acute{\pi}_{*}\acute{\mu}^{n})$ of qualitative robustness in
Definition \ref{def_qr_kyfan}. Conseuquently, the assumption of the pseudometric
$\acute{\alpha}_{\mu}$ amounts to Definition
 \ref{def_qr_kyfan} with  one epsilon instead of two corresponding to the
metric
\begin{equation}
\label{def_prpr}
d_{Pr^{2}\Theta}(\pi_{*}\mu^{n}, \acute{\pi}_{*}\acute{\mu}^{n}), \quad (\pi, \acute{\pi},\mu,\acute{\mu})
\in \mathcal{M}(\Theta)^{2} \times \mathcal{M}(X)^{2}\, .
\end{equation}
Moreover, non qualitative robustness with respect to this definition implies non qualitative
robustness with respect to the original  Definition \ref{def_qr_kyfan} with the Ky Fan metric.

Now lets turn to the proof that the inference is non qualitatively robust with respect to
the objective metric \eqref{def_prpr}.
Fix $\theta^{*} \in \Theta$
and consider another point
$\theta \in \Theta$ and  the Dirac mass $\delta_{\theta} \in \mathcal{M}(\Theta)$ situated at $\theta$.
For $\pi \in \mathcal{K}(\theta^{*})$,
 the  convex combination
\[\pi^{\alpha}:= \alpha \pi+(1-\alpha)\delta_{\theta}\]  is a probability measure with Kullback-Leibler support, that is,
$\pi^{\alpha}  \in \mathcal{K}(\theta^{*}),\, \alpha >0$ and
\begin{equation}
\label{eq_close}
d_{tv}(\pi^{\alpha},\delta_{\theta}) \leq \alpha.
\end{equation}
Therefore,   it follows that
\begin{equation*}
\label{eq_admiss} (\pi^{\alpha},\delta_{\theta}) \in \Pi_{\rho}(\theta^{*})\, ,\quad
 \alpha < \rho\, ,\,
\end{equation*}
and therefore
\begin{equation*}
\label{eq_admiss2}
 \bigl(\pi^{\alpha},\delta_{\theta},P_{\theta^{*}},P_{\theta^{*}}\bigr) \in Z_{\rho}(\theta^{*})\, ,\quad
\alpha <\rho \, ,
\end{equation*}
 where
$\mathcal{Z}_{\rho}(\theta^{*})$ is the admissible set defined
in \eqref{def_Z}.

For the prior $\pi^{\alpha}$, let
$\pi^{\alpha}_{n}:(X^{\infty},P_{\theta^{*}}^{\infty}) \rightarrow \mathcal{M}(\Theta)$, defined
by $\pi^{\alpha}_{n}(x^{\infty}):=\pi^{\alpha}_{x^{n}},\, x^{\infty} \in X^{\infty}$,
denote the corresponding sequence  of  posterior random variables,
and let
$(\pi^{\alpha}_{n})_{*}P_{\theta^{*}}^{\infty} \in \mathcal{M}^{2}(\Theta)$
denote its induced sequence of laws.
On the other hand,  for the prior $\delta_{\theta}$, it is easy to see that
 $(\delta_{\theta})_{x^{n}}=\delta_{\theta}, \, x^{n}\in X^{n}$, so that if we denote the corresponding sequence of
  posterior random variables
by $\delta^{n}_{\theta}$, then
 $(\delta^{n}_{\theta})_{*}P_{\theta^{*}}^{\infty}=(\delta_{\theta})_{*}P_{\theta^{*}}^{n}=\delta_{\delta_{\theta}}$.

Since the assumptions of Schwartz' Corollary \ref{cor_schwartz} are satisfied and $\pi^{\alpha}$ has Kullback-Leibler support at
$\theta^{*}$, we can apply the
assertion \eqref{conv_prob} of Proposition \ref{prop_schwartz}
\[P^{\infty}_{\theta^{*}}\Bigl\{ d_{Pr\Theta}(\pi^{\alpha}_{n},\delta_{\theta^{*}})> \epsilon\Bigr\} \rightarrow 0
\quad n \rightarrow \infty\, ,
\]
for $\epsilon >0$. To complete the proof we simply use the fact that
 convergence in law to a Dirac mass is  equivalent to convergence in probability to a constant random variable, that is
 use the equivalent
assertion \eqref{conv_proprob} of Proposition \ref{prop_schwartz}
\begin{equation}
\label{eq_conv}
d_{Pr^{2}\Theta}\Bigl((\pi^{\alpha}_{n})_{*}P_{\theta^{*}}^{\infty},\delta_{\delta_{\theta^{*}}}\Bigr)\rightarrow 0 \quad n \rightarrow \infty \, .
\end{equation}
 Now the proof is very simple.
Indeed,  from the triangle inequality we have
\[
d_{Pr^{2}\Theta}\Bigl((\pi^{\alpha}_{n})_{*}P_{\theta^{*}}^{\infty},\delta_{\delta_{\theta}}\Bigr)
\geq d_{Pr^{2}\Theta}\Bigl(\delta_{\delta_{\theta^{*}}},\delta_{\delta_{\theta}}\Bigr)-d_{Pr^{2}\Theta}\Bigl((\pi^{\alpha}_{n})_{*}P_{\theta^{*}}^{\infty},\delta_{\delta_{\theta^{*}}}\Bigr)
\]
and, by two applications of Proposition \ref{prop_delta},
we have
\begin{eqnarray*}
d_{Pr^{2}\Theta}\Bigl(\delta_{\delta_{\theta^{*}}},\delta_{\delta_{\theta}}\Bigr)&= & \min{\Bigl(d_{Pr\Theta}\Bigl(\delta_{\theta^{*}},\delta_{\theta}\Bigr)
, 1\Bigr)}\\
&= & \min{\Bigl(\min{\bigl(d(\theta^{*},\theta),1\bigr)}
, 1\Bigr)}\\
&= & \min{\bigl(d(\theta^{*},\theta),1\bigr)}\, .
\end{eqnarray*}
Therefore, since  $(\delta^{n}_{\theta})_{*}P_{\theta^{*}}^{\infty}=\delta_{\delta_{\theta}}$, the convergence
\eqref{eq_conv} implies that
\[d_{Pr^{2}\Theta}\Bigl((\pi^{\alpha}_{n})_{*}P_{\theta^{*}}^{\infty}, (\delta^{n}_{\theta})_{*}P_{\theta^{*}}^{\infty}\Bigr)
\rightarrow   \min{\bigl(d(\theta^{*},\theta),1\bigr)}\, ,\quad n \rightarrow \infty\,  .\]

Finally, since  $d_{Pr\Theta}\leq d_{tv}$,  it follows from
\eqref{eq_close} that
\begin{equation*}
\label{eq_closebis}
d_{Pr\Theta}(\pi^{\alpha},\delta_{\theta}) \leq \alpha.
\end{equation*}
Then,  for any $\delta>0$, if we  restrict $\alpha$ so that $\alpha < \min{(\delta,\rho)}$, it follows that
$ d_{tv}(\pi^{\alpha},\delta_{\theta})   <\rho$
and
$ d_{Pr\Theta}(\pi^{\alpha},\delta_{\theta})   <\delta,$
so that
\begin{equation}
\label{eq_admissbis}
\bigl(\pi^{\alpha},\delta_{\theta},P_{\theta^{*}},P_{\theta^{*}}\bigr) \in \mathcal{Z}_{\rho}(\theta^{*})\, ,
\end{equation}
 \begin{equation}
\label{eq_delta}
d_{Pr\Theta}(\pi^{\alpha},\delta_{\theta}) < \delta\, .
\end{equation}
Let $D:=\sup{\{d(\theta_{1},\theta_{2}): \theta_{1},\theta_{2} \in \Theta\} }$ denote the diameter of $\Theta$.
Then it follows from the triangle inequality that, for any $\epsilon >0$, there exists a $\theta \in \Theta$ such that
$d(\theta^{*},\theta) \geq \frac{D}{2}-\epsilon$.
Consequently, for any  $\bar{\epsilon} <\min{(\frac{D}{2},1)}  $,
  no matter how small $\delta$ is, there is an $\alpha>0$ such that, in addition to \eqref{eq_admissbis} and \eqref{eq_delta},
  we have
\begin{equation*}
\label{eq_nonrb}
d_{Pr^{2}\Theta}\Bigl((\pi^{\alpha}_{n})_{*}P_{\theta^{*}}^{\infty}, (\delta^{n}_{\theta})_{*}P_{\theta^{*}}^{\infty}\Bigr)
>  \bar{\epsilon}\,,
\end{equation*}
for large enough $n$.
Consequently, the assertion is proved.

\subsection{Proof of Theorem \ref{thm_main3}}
As in the proof of Theorem \ref{thm_main}, we establish the
assertion with respect to the modified form of qualitative robustness defined by \eqref{def_prpr},
 and since this form is weaker it implies
the assertion.
  It follows from the definition of the packing numbers that, for $\epsilon >0$, there is a
packing  $\{\theta_{i},i=1,..,\mathcal{M}_{2\epsilon}(\Theta)\}$ and therefore  the collection
of open balls $B_{\epsilon}(\theta_{i}),\, i=1,..,\mathcal{M}_{2\epsilon}(\Theta)$ is a disjoint union. Denoting
$\mathcal{N}_{2\epsilon}:=\mathcal{N}_{2\epsilon}(\Theta)$ and $\mathcal{M}_{2\epsilon}:=\mathcal{M}_{2\epsilon}(\Theta)$,
  we therefore obtain
\begin{eqnarray*}
1&=& \pi(\Theta)\\
&\geq & \pi\bigl(\cup_{i=1}^{\mathcal{M}_{2\epsilon}}{B_{\epsilon}(\theta_{i})}\bigr)\\
&= & \sum_{i=1}^{\mathcal{M}_{2\epsilon}}{\pi\bigl(B_{\epsilon}(\theta_{i})\bigr)}\\
&\geq & \mathcal{M}_{2\epsilon}\min_{i=1,\mathcal{M}_{2\epsilon}}{\pi\bigl(B_{\epsilon}(\theta_{i})\bigr)}\, .
\end{eqnarray*}
 Consequently, since \eqref{eq_kolm} implies $\mathcal{M}_{2\epsilon} \geq \mathcal{N}_{2\epsilon}$, there exists a point $\theta^{*}\in \Theta$ such that
\begin{equation}
\label{eq_N}
\pi\bigl(B_{\epsilon}(\theta^{*})\bigr) \leq \frac{1}{\mathcal{N}_{2\epsilon}} \, .
\end{equation}
 Let  $B_{\epsilon}:=B_{\epsilon}(\theta^{*})$ denote the open ball about $\theta^{*}$ and
let $B_{\epsilon}^{c}$ denote its complement. Let $\pi^{\epsilon} \in \mathcal{M}(\Theta)$, defined by
  \[\pi^{\epsilon}(B):=\frac{\pi(B^{c}_{\epsilon}\cap B)}{\pi(B^{c}_{\epsilon})},\,  B \in \mathcal{B}(\Theta)\, ,\]
denote the normalization of the restriction of $\pi$  to $B^{c}_{\epsilon}$ which,
by the inequality \eqref{eq_N}, is well defined.
Since
$ \pi=\pi(B^{c}_{\epsilon})\pi^{\epsilon}+\pi|_{B_{\epsilon}}$
it follows that
$\pi-\pi^{\epsilon}= \pi|_{B_{\epsilon}}- \pi(B_{\epsilon})\pi^{\epsilon}
$
so that we obtain
\[  d_{tv}(\pi^{\epsilon}, \pi) \leq \pi(B_{\epsilon}) \leq \frac{1}{\mathcal{N}_{2\epsilon}}\]
from which we obtain
\begin{equation}
\label{eq_close3bis}
d_{Pr\Theta}(\pi^{\epsilon},\pi) \leq \frac{1}{\mathcal{N}_{2\epsilon}}\, .
\end{equation}
In particular, when $ \frac{1}{\mathcal{N}_{2\epsilon}} < \rho$,
 we obtain
\begin{equation*}
\label{eq_admiss8bis}
 \pi^{\epsilon} \in B^{tv}_{\rho}(\pi)
\end{equation*}
and therefore
\begin{equation*}
\label{eq_admisbis}
 \bigl(\pi, \pi^{\epsilon},P_{\theta^{*}},P_{\theta^{*}}\bigr) \in Z_{\rho}(\pi)\,  .
\end{equation*}
That is, when $ \frac{1}{\mathcal{N}_{2\epsilon}} < \rho$, the point
$\bigl(\pi, \pi^{\epsilon},P_{\theta^{*}},P_{\theta^{*}}\bigr) \in Z_{\rho}(\pi)$.

For the prior $\pi^{\epsilon}$, let
$\pi^{\epsilon}_{n}:(X^{\infty},P_{\theta^{*}}^{\infty}) \rightarrow \mathcal{M}(\Theta)$, defined
by $\pi^{\epsilon}_{n}(x^{\infty}):=\pi^{\epsilon}_{x^{n}},\, x^{\infty} \in X^{\infty}$,
denote the corresponding sequence  of  posterior random variables,
and let
$(\pi^{\epsilon}_{n})_{*}P_{\theta^{*}}^{\infty} \in \mathcal{M}^{2}(\Theta)$
denote its induced sequence of laws.
Since the assumptions of Schwartz' Corollary \ref{cor_schwartz} are satisfied and $\pi$ has Kullback-Leibler support at
$\theta^{*}$, we can apply the
assertion \eqref{conv_proprob} of Proposition \ref{prop_schwartz} to the sequence of posterior laws
 $(\pi_{n})_{*}P_{\theta^{*}}^{\infty}$ corresponding to $\pi$:
\begin{equation}
\label{eq_conv2}
d_{Pr^{2}\Theta}\Bigl((\pi_{n})_{*}P_{\theta^{*}}^{\infty},\delta_{\delta_{\theta^{*}}}\Bigr)\rightarrow 0 \quad n \rightarrow \infty \, .
\end{equation}
From the triangle inequality we have
\begin{equation}
\label{eq_triangle}
d_{Pr^{2}\Theta}\Bigl((\pi_{n})_{*}P_{\theta^{*}}^{\infty},(\pi^{\epsilon}_{n})_{*}P_{\theta^{*}}^{\infty}\Bigr)
\geq d_{Pr^{2}\Theta}\Bigl((\pi^{\epsilon}_{n})_{*}P_{\theta^{*}}^{\infty}, \delta_{\delta_{\theta^{*}}}\Bigr)-d_{Pr^{2}\Theta}\Bigl((\pi_{n})_{*}P_{\theta^{*}}^{\infty},\delta_{\delta_{\theta^{*}}}\Bigr)\,,
\end{equation}
so to lower bound the lefthand side it is sufficient in the limit to  lower bound
the first term on the right. To that end,
 we use
a quantitative  version of the partial  converse~\cite[Thm.~11.3.5]{Dudley:2002} of convergence in probability implies convergence in law, valid when
the convergence in law is to a Dirac mass.
 Indeed, if we denote the Ky Fan metric determined from the measure $P_{\theta^{*}}^{\infty}$ by $\alpha_{\theta^{*}}$, Lemma \ref{lem_kyfanprokhorov} asserts that
\begin{equation}
\label{eq_kf_pr}
d_{Pr^{2}\Theta}\Bigl((\pi^{\epsilon}_{n})_{*}P_{\theta^{*}}^{\infty}, \delta_{\delta_{\theta^{*}}}\Bigr)=\alpha_{\theta^{*}}(\pi^{\epsilon}_{n},\delta_{\theta^{*}})\, .
\end{equation}
To evaluate the Ky Fan distance on the righthand side,
first observe that since $\pi^{\epsilon}$ has support contained
in the closed set $B^{c}_{\epsilon}$, it follows from  Schervish \cite[Thm.~1.31]{Schervish_theory} that
$\pi^{\epsilon}_{x^{n}} $ also has support contained in $B^{c}_{\epsilon}$ a.e $P_{\theta^{*}}^{n}$.
Therefore, if we define $B_{0}:=\{\theta^{*}\}$ and $B_{r}:=B_{r}(\theta^{*})$, it follows that
$B_{0}^{r}=B_{r}$, so that
\begin{equation*}
\label{eq_xx}
\pi^{\epsilon}_{x^{n}}(B_{0}^{r})=0\, ,\quad  a.e.~P_{\theta^{*}}^{n},\quad  r < \epsilon
\end{equation*}
and
\begin{equation*}
\label{eq_yy}
(\delta_{\theta^{*}})_{x^{n}}(B_{0})=1\, ,\quad  a.e.~P_{\theta^{*}}^{n}\, .
\end{equation*}
It follows from  Lemma \ref{lem_lb} that
\begin{eqnarray*}
 d_{Pr\Theta}\bigl(\pi^{\epsilon}_{x^{n}},(\delta_{\theta^{*}})_{x^{n}}\bigr)
& \geq & \min{(\epsilon, 1)}\quad  a.e.~P_{\theta^{*}}^{\infty}\, ,
\end{eqnarray*}
and, since $\epsilon \leq 1$, we obtain
\begin{equation*}
\label{eq_zz}
P_{\theta^{*}}^{\infty}\Bigl(d_{Pr\Theta}\bigl(\pi^{\epsilon}_{x^{n}},(\delta_{\theta^{*}})_{x^{n}}\bigr)\geq  \epsilon \Bigr)=1\, .
\end{equation*}
Therefore, by the definition
\eqref{def_kyfan} of the Ky Fan metric, we obtain
 $\alpha_{\theta^{*}}(\pi^{\epsilon}_{n},\delta_{\theta^{*}}) \geq \epsilon$ and, by the identity
\eqref{eq_kf_pr}, we conclude that
\[d_{Pr^{2}\Theta}\Bigl((\pi^{\epsilon}_{n})_{*}P_{\theta^{*}}^{\infty}, \delta_{\delta_{\theta^{*}}}\Bigr)\geq \epsilon  \, .\]
Consequently, from the triangle inequality \eqref{eq_triangle} and the convergence \eqref{eq_conv2},
we  conclude, for any $\acute{\epsilon}>0$,   that for  large enough $n$ we have
\begin{equation}
\label{eq_dddd}
d_{Pr^{2}\Theta}\Bigl((\pi_{n})_{*}P_{\theta^{*}}^{\infty},(\pi^{\epsilon}_{n})_{*}P_{\theta^{*}}^{\infty}\Bigr) \geq \epsilon -\acute{\epsilon}\, .
\end{equation}
Consequently, if this Bayesian inference is qualitatively robust, then for $\epsilon>0$, it follows from
\eqref{eq_dddd} and  \eqref{eq_close3bis}  that
$\delta < \frac{1}{\mathcal{N}_{2\epsilon}}$. The requirement that perturbations be admissible, that is determine members in $\mathcal{Z}_{\rho}(\pi)$, implies that $\delta < \rho$.

\section{Appendix}

\subsection{Schwartz' Theorem and the convergence of random measures}
\label{sec_schwartzfacts}
It will be useful to
 express the assertion of Corollary \ref{cor_schwartz} and some of its consequences in
terms of  the  convergence of measures and random measures. To that end, recall the  notation
$\mathcal{M}^{2}(\Theta):=\mathcal{M}(\mathcal{M}(\Theta))$, and
consider the corresponding sequence of random variables
$\pi_{n}:(X^{\infty},P_{\theta^{*}}^{\infty}) \rightarrow \mathcal{M}(\Theta)$, defined
by $\pi_{n}(x^{\infty}):=\pi_{x^{n}},\, x^{\infty} \in X^{\infty}$,
and its induced sequence of laws
$(\pi_{n})_{*}P_{\theta^{*}}^{\infty} \in \mathcal{M}^{2}(\Theta)$.
Note especially that
 $\delta_{\delta_{\theta^{*}}}$ is the Dirac mass in $\mathcal{M}^{2}(\Theta)$
situated at the Dirac mass $\delta_{\theta^{*}}$ in $\mathcal{M}(\Theta)$ situated at $\theta^{*}$.

\begin{prop}
\label{prop_schwartz}
The assertion of Corollary \ref{cor_schwartz}
 is equivalent to
\begin{equation*}
\label{con_as} \pi_{x^{n}} \mapsto  \delta_{\theta^{*}} \quad   a.e.~P^{\infty}_{\theta^{*}}\, ,
\end{equation*}
where $\mapsto$ is weak convergence.
This in turn implies  that
\begin{equation}
\label{conv_prob}
P^{\infty}_{\theta^{*}}\Bigl\{ d_{Pr\Theta}(\pi_{n},\delta_{\theta^{*}})> \epsilon\Bigr\} \rightarrow 0
\quad n \rightarrow \infty\, ,
\end{equation}
for $\epsilon >0$,
which is equivalent to
\begin{equation}
\label{conv_proprob}
d_{Pr^{2}\Theta}\Bigl((\pi_{n})_{*}P_{\theta^{*}}^{\infty},\delta_{\delta_{\theta^{*}}}\Bigr)\rightarrow 0 \quad n \rightarrow \infty \, ,
\end{equation}
where $d_{Pr^{2}\Theta}$ is the Prokhorov metric on $\mathcal{M}^{2}(\Theta)$ defined with respect to the Prokhorov metric $d_{Pr\Theta}$ on $\mathcal{M}(\Theta)$.
\end{prop}

\begin{proof}
Let $\mathcal{O}$ denote the open sets in $\Theta$ and
 $\mathcal{O}_{\theta^{*}}\subset \mathcal{O}$ denote the open neighborhoods of $\theta^{*}$.   Then,
under the conditions of Corollary \ref{cor_schwartz},  for $O \in \mathcal{O}_{\theta^{*}}$, it follows that
\[ \pi_{x^{n}}(O) \rightarrow 1 \quad  n\rightarrow \infty,\quad   a.e.~P^{\infty}_{\theta^{*}}\, .\]
Since $\delta_{\theta^{*}}(O)=1, \, O \in \mathcal{O}_{\theta^{*}}$ and
$\delta_{\theta^{*}}(O)=0,\,  O \in   \mathcal{O}\setminus \mathcal{O}_{\theta^{*}}$  it easily follows that
 \[ \lim \inf_{n}\pi_{x^{n}}(O) \geq \delta_{\theta^{*}}(O),\,\, \forall O \in \mathcal{O},  \quad   a.e.~P^{\infty}_{\theta^{*}}\, .\]
which, by the Portmanteau theorem \cite[Thm.~11.1.1]{Dudley:2002}, is equivalent to
\[ \pi_{x^{n}} \mapsto  \delta_{\theta^{*}} \quad   a.e.~P^{\infty}_{\theta^{*}}\, .\]
where $\mapsto$ denotes weak convergence.

Now consider the corresponding sequence of random variables
$\pi_{n}:(X^{\infty},P_{\theta^{*}}^{\infty}) \rightarrow \mathcal{M}(\Theta)$, defined
by $\pi_{n}(x^{\infty}):=\pi_{x^{n}},\, x^{\infty} \in X^{\infty}$,
and its induced sequence of laws
$(\pi_{n})_{*}P_{\theta^{*}}^{\infty} \in \mathcal{M}^{2}(\Theta)$.
Then $\pi_{x^{n}} \mapsto  \delta_{\theta^{*}}\,   a.e.~P^{\infty}_{\theta^{*}}$
is equivalent to
\[ \pi_{n} \mapsto  \delta_{\theta^{*}} \quad  a.s.~P^{\infty}_{\theta^{*}}\, .\]
Since $\Theta$ is a separable metric space it follows that $\mathcal{M}(\Theta)$ equipped with the Prokhorov metric
is a separable metric space. Since a.s.~convergence implies convergence in probability for random variables with values
in a separable metric space, it follows that
\[\pi_{n} \mapsto  \delta_{\theta^{*}}\, \text{in}\, P^{\infty}_{\theta^{*}}-\text{probability}\, ,\]
that is,
\[P^{\infty}_{\theta^{*}}\Bigl\{ d_{Pr\Theta}(\pi_{n},\delta_{\theta^{*}})> \epsilon\Bigr\} \rightarrow 0
\quad n \rightarrow \infty\, .
\]

Since $\mathcal{M}(\Theta)$ is a separable metric space it follows that $\mathcal{M}^{2}(\Theta)$ equipped
with the Prokhorov metric is also a separable metric space.
Therefore, since on
separable metric spaces convergence in probability to a constant valued random variable is equivalent
to the weak convergence of the corresponding set of laws to the Dirac mass situated at that value,
see e.g.~Dudley \cite[Prop.~11.1.3]{Dudley:2002}, it follows that the convergence in probability,
$\pi_{n} \rightarrow  \delta_{\theta^{*}}\, \text{in}\, P^{\infty}_{\theta^{*}}-\text{probability}$,
 is equivalent to
the corresponding convergence of laws
\[ (\pi_{n})_{*}P_{\theta^{*}}^{\infty} \mapsto \delta_{\delta_{\theta^{*}}}\quad  n \rightarrow \infty\, .\]
Finally, since the Proprokhorov metric $d_{Pr^{2}\Theta}$ on $\mathcal{M}^{2}(\Theta)$ metrizes the weak topology on $ \mathcal{M}^{2}(\Theta)=
\mathcal{M}(\mathcal{M}(\Theta))$,
it follows that the latter is equivalent to
\[d_{Pr^{2}\Theta}\Bigl((\pi_{n})_{*}P_{\theta^{*}}^{\infty},\delta_{\delta_{\theta^{*}}}\Bigr)\rightarrow 0 \quad n \rightarrow \infty \, .
\]

\end{proof}

\subsection{Some Prokhorov Geometry}
\label{sec_appendix}
We establish a basic mechanism  to  bound from below
the  Prokhorov distance between two measures based on the values of the measures on the neighborhood of a single set.
\begin{lem}
\label{lem_lb}
Let $Z$ be a metric space
 and consider  the space $\mathcal{M}(Z)$ of Borel probability
measures equipped with the
 Prokhorov metric.
Consider $\mu \in \mathcal{M}(Z)$ and suppose that there exists a set $B \in \mathcal{B}(Z)$
  and $\alpha, \delta \geq 0$ such that
\[ \mu(B^{\epsilon}) \leq \delta, \quad \epsilon < \alpha\, .\]
Then, for any $\mu' \in \mathcal{M}(Z)$, we have
\[d_{Pr}(\mu,\mu') \geq \min{\bigl( \alpha, \mu'(B)-\delta \bigr)}\, .\]
\end{lem}
\begin{proof}
If $d_{Pr}(\mu_{1},\mu_{2}) \geq \alpha$  the assertion is proved, so let us assume that
$d_{Pr}(\mu_{1},\mu_{2}) <  \alpha$. Then, denoting  $d^{*}:=d_{Pr}(\mu_{1},\mu_{2})$, it follows from the assumption
that $ \mu(A^{d^{*}}) \leq \delta $, so that
\begin{eqnarray*}
\mu'(A)& \leq& \mu(A^{d^{*}}) +d^{*}\\
&\leq & \delta +d^{*}
\end{eqnarray*}
from which we conclude that
$ \mu'(A)-\delta \leq d^{*}$.
Therefore, either $d_{Pr}(\mu_{1},\mu_{2}) \geq \alpha$  or $d_{Pr}(\mu_{1},\mu_{2}) \geq  \mu'(A)-\delta $,
proving the assertion.
\end{proof}

\begin{lem}
\label{lem_kyfanprokhorov}
Let $S$ be a separable metric space. Then, for  an $S$-valued random variable $X$ we have
\[ \alpha(X,s)= d_{Pr}(\mathcal{L}(X),\delta_{s})\]
where $\alpha$ is the Ky Fan metric and $s$ denotes the random variable with constant value $s$.
\end{lem}
\begin{proof}
Let us denote $\alpha:=\alpha(X,s)$ and
$\rho:=d_{Pr}(\mathcal{L}(X),\delta_{s})$.  Define the set $B_{0}:=\{s\}$ and $B_{r}:=B_{r}(s), r >0$ and observe
that $B_{0}^{r}=B_{r}, r >0$. Therefore, by the definition of $\rho$ we have
\[\mathcal{L}(s)(B_{0}) \leq \mathcal{L}(X)(B^{\rho}_{0}) +\rho\]
and since $\mathcal{L}(s)(B_{0}) =1$ we obtain
\[ \mathcal{L}(X)(B^{\rho}_{0}) \geq 1-\rho\] from which we obtain
$ P( d(X,s) \geq \rho) \leq \rho\, .$
Since this implies that
\[P( d(X,s) > \rho) \leq  P( d(X,s) \geq \rho) \leq \rho \, \]
we conclude that $\rho \leq \alpha$. Since
 Dudley \cite[Thm.~11.3.5]{Dudley:2002} asserts that
$\alpha \leq \rho$, the assertion follows.
\end{proof}

\begin{prop}
\label{prop_delta}
\[d_{Pr}(\delta_{x_{1}},\delta_{x_{2}}) = \min{\bigl(1, d(x_{1},x_{2})\bigr)}\]
\end{prop}
\begin{proof}
Consider the set $B:=\{x_{1}\}$. Then since $ B^{\epsilon}=B_{\epsilon}(x_{1})$, it follows that for $\epsilon < d(x_{1},x_{2})$  that
$x_{2} \notin B^{\epsilon}$. Consequently, since $\delta_{x_{1}}(B)=1$, the inequality
\[\delta_{x_{1}}(B)\leq \delta_{x_{2}}(B^{\epsilon})+\epsilon\] requires either
$\epsilon \geq 1$ or $x_{2} \in B^{\epsilon}$ which implies that $\epsilon \geq d(x_{1},x_{2})$.
Consequently, $d_{Pr}(\delta_{x_{1}},\delta_{x_{2}}) \geq \min{\bigl(1, d(x_{1},x_{2})\bigr)}$. To obtain equality,
suppose that
$d_{Pr}(\delta_{x_{1}},\delta_{x_{2}}) >  d(x_{1},x_{2})$. Then, for any $d'$  which satisfies
$d_{Pr}(\delta_{x_{1}},\delta_{x_{2}}) > d'>  d(x_{1},x_{2})$ there exists a  measurable set $B$ such that
\[\delta_{x_{1}}(B)> \delta_{x_{2}}(B^{d'})+d'\, \]
Consequently, $x_{1} \in B$, but $d' >  d(x_{1},x_{2})$ implies that $x_{2} \in B^{d'}$, which implies
the contradiction $1 > 1+d'$.
\end{proof}

\section*{Acknowledgments}
The authors gratefully acknowledges this work supported by  the Air Force Office of Scientific Research and the DARPA EQUiPS Program under
awards number  FA9550-12-1-0389 (Scientific Computation of Optimal Statistical Estimators) and number FA9550-16-1-0054 (Computational Information Games).

\addcontentsline{toc}{section}{Acknowledgments}


\addcontentsline{toc}{section}{References}
\bibliographystyle{plain}
\bibliography{refs}

\end{document}